\documentclass[11pt]{article}
\usepackage{a4}
\usepackage{amsthm, amsmath, amsfonts, cite, enumitem}
\usepackage{graphicx}
\usepackage{caption}
\usepackage{titlesec}
\graphicspath{ {./Paperfigures/} }
\usepackage{epsfig}
\usepackage{caption}  
\usepackage{capt-of}
\usepackage{tikz}
\usepackage{enumitem}
\usepackage{multicol}
\usepackage{thmtools}
\usepackage{thm-restate}
\usepackage{cleveref}
\newtheorem{theorem}{Theorem}[section]

\newtheorem{lemma}[theorem]{Lemma}

\title{Vertex-minor-closed classes are $\chi$-bounded}
\author{James Davies\thanks{Department of Combinatorics and Optimization, University of Waterloo, Waterloo, Canada. E-mail: \texttt{jgdavies@uwaterloo.ca}.}}
\date{}

\begin{document}

\maketitle

%\classno{05C15 (primary), 05C62 (secondary).}

\begin{abstract}
	We prove a conjecture of Geelen that every proper vertex-minor-closed class of graphs is $\chi$-bounded.
\end{abstract}

\section{Introduction}

A class of graphs $\mathcal{G}$ is \emph{$\chi$-bounded} if graphs in $\mathcal{G}$ with bounded clique number also have bounded chromatic number.
Geelen (see~\cite{dvovrak2012classes}) conjectured that every proper vertex-minor-closed class of graphs is $\chi$-bounded (we delay certain definitions such as that of vertex-minors until Section~2). We prove this conjecture.

\begin{theorem}\label{main}
	Every proper vertex-minor-closed class of graphs is $\chi$-bounded.
\end{theorem}

Scott~\cite{scott1997induced} conjectured that for every graph $H$, the class of graphs containing no induced subdivision of $H$ is $\chi$-bounded. This conjecture was disproved by Pawlik et al~\cite{pawlik2014triangle}. However if a graph contains an induced subdivision of a graph $H$, then it also contains $H$ as a vertex-minor. So Theorem~\ref{main} recovers one possible weakening of Scott's conjecture.
For further results and conjectures on $\chi$-boundedness, there is an excellent recent survey by Scott and Seymour~\cite{scott2018survey}.

Several special cases of Theorem~\ref{main} have been proved in the past. Most classically, Gy\'{a}rf\'{a}s~\cite{gyarfas1985chromatic} proved that circle graphs are $\chi$-bounded. Another important vertex-minor-closed class of graphs is those with bounded rank-width; Dvo{\v{r}}{\'a}k and Kr\'{a}l'~\cite{dvovrak2012classes} proved that such graphs are $\chi$-bounded. Geelen, Kwon, McCarty and Wollan~\cite{geelen2019grid} proved that if $H$ is a circle graph, then the class of graphs with no $H$ vertex-minor has bounded rank-width and so is also $\chi$-bounded. Let $W_n$ denote the wheel graph consisting of an $n$-cycle and a single additional dominating vertex. Two of the three minimal forbidden vertex-minors for circle graphs are wheel graphs~\cite{bouchet1994circle}, and so generalising Gy\'{a}rf\'{a}s's result that circle graphs are $\chi$-bounded, Choi, Kwon, Oum and Wollan~\cite{choi2019chi} proved that the for each $n$, the class of graphs with no $W_n$ vertex-minor are $\chi$-bounded. Kostochka~\cite{kostochka1988upper} proved that the complements of circle graphs are $\chi$-bounded. This class of graphs is not vertex-minor-closed, however its closure under vertex-minors can be shown to be $\chi$-bounded as an extension of Kostochka's result~\cite{Jim}.

We also prove a natural weakening of $\chi$-boundedness but with a linear bound.
For a graph $G$ and a positive integer $\rho$, a \emph{$\rho$-ball} is an induced subgraph formed by a vertex $v$ and the vertices at distance at most $\rho$ from $v$.
We let $\chi^{(\rho)}(G)$ denote the maximum chromatic number of a $\rho$-ball contained in a graph $G$, and we say that a class of graphs $\mathcal{G}$ is {\emph{$\rho$-controlled}} if there exists a function $f$ such that $\chi(G)\le f(\chi^{(\rho)}(G))$ for all $G\in {\mathcal{G}}$. A class of graphs $\mathcal{G}$ is \emph{linearly $\rho$-controlled} if there exists a constant $c$ such that $\chi(G)\le c\chi^{(\rho)}(G)$ for all $G\in {\mathcal{G}}$.

With the idea of $\rho$-control in mind, one may naturally split the problem of proving that a class of graphs $\mathcal{G}$ is $\chi$-bounded into subproblems. The first is to show that for some $\rho \ge 2$, $\mathcal{G}$ is $\rho$-controlled. The next is to reduce control and show that $\mathcal{G}$ is 2-controlled. The final subproblem is to make use of the fact that $\mathcal{G}$ is 2-controlled to prove $\chi$-boundedness.

We follow this strategy to prove Theorem~\ref{main} and along the way prove that proper vertex-minor-closed classes of graphs are in fact linearly 9-controlled.

\begin{theorem}\label{liner 9-control}
	Every proper vertex-minor-closed class of graphs is linearly \allowbreak 9-controlled.
\end{theorem}

To go from 9-control to 2-control we shall simply apply a theorem of Chudnovsky, Scott and Seymour~\cite{chudnovsky2016induced}. Unfortunately this theorem does not preserve linearity if the class of graphs has unbounded clique number. However we conjecture that every proper vertex-minor-closed class of graphs is linearly 2-controlled.
Note that in general not all vertex-minor-closed classes of graphs have linear $\chi$-bounding functions, or are even linearly 1-controlled. For instance Kostochka~\cite{kostochka1988upper,kostochka2004coloring} showed that no linear $\chi$-bounding function exists for circle graphs, while 1-balls contained in circle graphs are permutation graphs, which are perfect (see~\cite{golumbic2004algorithmic}).
For more on the notion of $\rho$-control, see~\cite{scott2018survey}.

Building on Geelen's conjecture, Kim, Kwon, Oum and Sivaraman~\cite{kim2020classes} further asked if all proper vertex-minor-closed classes of graphs are polynomially $\chi$-bounded.
Recently there have also been significant developments on this problem. The author and McCarty~\cite{davies2019circle} proved a quadratic $\chi$-bounding function for circle graphs and Bonamy and Pilipczuk~\cite{bonamy2019graphs} proved that graphs of bounded rank-width are polynomially $\chi$-bounded. As a result, it also follows that if $H$ is a circle graph, then the class of graphs with no $H$ vertex-minor are polynomially $\chi$-bounded~\cite{bonamy2019graphs,geelen2019grid}.

Much less is known for pivot-minor-closed classes of graphs. However again circle graphs~\cite{gyarfas1985chromatic,kostochka1988upper,davies2019circle} and graphs of bounded rank-width~\cite{dvovrak2012classes,bonamy2019graphs} are both {$\chi$-bounded} and closed under pivot-minors. In addition, Choi, Kwon and Oum~\cite{choi2017coloring} proved that for each $n$, the class of graphs containing no $n$-cycle as a pivot-minor is $\chi$-bounded.
More recently Scott and Seymour~\cite{scott2019induced} proved a significant generalisation of this, that for all integers $k \ge 0$ and $\ell \ge  1$, the class of all graphs with no induced cycle of length $k$ modulo $\ell$ is $\chi$-bounded.
This implies the result of Choi, Kwon and Oum as a cycle of length $n+2$ contains a cycle of length $n$ as a pivot-minor.

We will also make a first step towards proving the conjecture of Choi, Kwon and Oum~\cite{choi2017coloring} that proper pivor-minor-closed classes of graphs are $\chi$-bounded. Following the idea of $\rho$-control, the first step we make is the following:

\begin{theorem}\label{pivot step}
	Every pivot-minor-closed class of graphs that is 2-controlled is also $\chi$-bounded.
\end{theorem}

This reduces the problem of proving that a pivot-minor-closed class of graphs is $\chi$-bounded to proving that it is 2-controlled.

\section{Preliminaries}

Given two sets $A$ and $B$, we let $A-B$ denote the subset of $A$ obtained by removing the elements of $A\cap B$.

Given a vertex $v$ of a graph $G$, we let $N(v)$ denote its \emph{neighbourhood}, i.e., the set of vertices adjacent to $v$.
More generally given a set of vertices $A$ of a graph $G$, we let $N(A)$ be the set of vertices in $V(G)-A$ that are adjacent to a vertex of $A$.
If the graph is not clear from context, we use $N_G(v)$ or $N_G(A)$.
Given an integer $t\ge 0$, we let $N_t(A)$ be the set of vertices at distance exactly $t$ from $A$, and we let $N_{t}[A]$ be the set of vertices at distance at most $t$ from $A$. We may denote the \emph{closed neighbourhood} $N_1[A]$ by $N[A]$.

A \emph{clique} in a graph $G$ is a set of pairwise adjacent vertices.
The \emph{clique number} $\omega(G)$ of $G$ is equal to the size of the largest clique contained in $G$. A \emph{stable set} is a set of pairwise non-adjacent vertices.
For positive integers $n,m$, we let $R(n,m)$ denote the \emph{Ramsey number} of $(n,m)$, i.e., all graphs on at least $R(n,m)$ vertices contain either a clique of size $n$ or a stable set of size $m$.

We say that two sets of vertices $A$ and $B$ in a graph $G$ are \emph{complete} to each other if for all $a\in A$ and $b\in B$, we have $ab\in E(G)$. Similarly $A$ and $B$ are \emph{anti-complete} if for all $a\in A$ and $b\in B$, we have $ab \not\in E(G)$. If for all $b\in B$, there exists a vertex $a\in A$ that is adjacent to $b$, then we say that $A$ \emph{dominates} $B$. For a simple example observe that if $v$ is a vertex of a graph $G$, then $N_{t-1}(v)$ dominates $N_t(v)$.
For a positive integer $n$, we let $[n]$ denote the set $\{1, 2, \dots ,n\}$.

Given a set of vertices $C$ of a graph $G$, we denote the \emph{induced subgraph} of $G$ on vertex set $C$ by $G[C]$. For convenience we often use $\chi(C)$ for $\chi(G[C])$. Given a set $A$ of vertices of a graph $G$, we let $G-A$ be the graph obtained from $G$ by deleting the vertices $A$. Similarly for a set $E$ of edges of $G$, we let $G$ be the graph obtained from $G$ by deleting the edges $E$. For a set $E$ edges in $G$, the graph obtained from $G$ by contracting each edge of $E$ (and then removing any resulting loops or multiple edges) is denoted by $G/E$. Given two disjoint sets $A$ and $B$ of vertices in a graph $G$, we let $E(A,B)$ denote the set of edges between $A$ and $B$. For a set $A$ of vertices in a graph $G$, we let $E(A)$ denote the set of edges between vertices of $A$.

The action of performing \emph{local complementation} at a vertex $v$ in a graph $G$ replaces the induced subgraph on
$N(v)$ by its complement. We denote the resulting graph by $G * v$. We
say that a graph $H$ is a vertex-minor of a graph $G$ if $H$ can be obtained from $G$ by a sequence of vertex deletions and local complementations.

In a graph $G$ the operation of \emph{pivoting} an edge $uv$ is: $G \land uv = G * u * v * u$. A graph $H$ is a \emph{pivot-minor} of a graph $G$ if $H$ can be obtained from $G$ by a sequence of vertex deletions and pivots. For an edge $uv$ of a graph $G$, let $V_1=N(u)-N[v]$, $V_2=N(v)-N[u]$, and $V_3=N(u) \cap N(v)$. It is straightforward to show that $G\wedge uv$ is the graph obtained from $G$ by first complementing the edges between each the three pairs of vertex sets $(V_1,V_2), (V_2,V_3)$, and $(V_1,V_3)$, and then swapping the vertex labels of $u$ and $v$. We will use this often and without explicit reference; for a formal proof, see~\cite{oum2005rank}.

In a graph $G$, the act of replacing an edge $uw$ with a vertex $v$ adjacent to $u$ and $w$ only is known as \emph{subdividing} the edge $uw$. A graph $H$ is a \emph{subdivision} of a graph $G$ if $H$ can be
obtained from $G$ by a sequence of subdivisions. We let $G^k$ denote the graph obtained from $G$ by subdividing each edge $k$ times.

If $v$ is a vertex of degree two in a graph $G$ and $v$ is adjacent to two non-adjacent vertices $u$ and $w$ then we say that the graph obtained by removing the vertex $v$ and adding an edge between $u$ and $w$ is the graph obtained from $G$ by \emph{smoothing} the vertex $v$. Observe that the graph obtained from $G$ by smoothing a vertex $v$ is $(G * v)-v$, and so in particular is a vertex-minor of $G$.
So more generally, by repeated smoothing of vertices, a graph $G$ is a vertex-minor of any subdivision of $G$.

For positive integers $n,m$, we let $K_{n,m}$ denote the complete bipartite graph whose vertices can be partitioned into two stable sets of size $n$ and $m$ that are complete to each other. We prove a motivating lemma.

\begin{lemma}\label{K^1_n,n}
	The graph $K_{n,{n \choose 2}}^1$ contains all $n$-vertex graphs as vertex-minors.
\end{lemma}

\begin{proof}
	First observe that by smoothing and deleting vertices we may obtain $K_n^1$ as a vertex-minor.
	Then it can easily be checked for $n\le 3$ that $K_n^1$ contains every $n$-vertex graph as a vertex-minor by smoothing and deleting vertices.
	So we may assume that $n \ge 4$.
	Now given an $n$-vertex graph $G$, we may associate its vertices with the vertices of degree at least 3 in $K_n^1$. Now for each pair of distinct vertices $u$ and $v$ of $G$ we may do one of two things. If $uv$ is an edge of $G$ then we may simply smooth the corresponding degree-2 vertex of $K_n^1$. If $uv$ is not an edge then we may just delete the corresponding degree-2 vertex of $K_n^1$. Doing this for each such pair $u,v$ results in the desired vertex-minor $G$.
\end{proof}

So $K_{n,n}^1$ provides a suitable universal graph for vertex-minors.
We take Lemma~\ref{K^1_n,n} a step further. A graph $G$ with vertex-set $\{x_1,\dots, x_n\}\cup \{y_1,\dots , y_m\} \cup \{z_{i,j}: i\in [n], j\in [m]\}$ is an \emph{interfered $K_{n,m}^1$} if
\begin{itemize}
	\item for each $i\in [n], j\in [m]$, $x_iz_{i,j}\in E(G)$ and $z_{i,j}y_j\in E(G)$,
	
	\item all other edges of $G$ are contained in ${\{x_kz_{i,j}: i,k\in [n],j\in [m] \text{ with } k<i\}}$.
\end{itemize}

An interfered $K_{n,m}^1$ is \emph{completely interfered} if $x_kz_{i,j}$ is an edge for all $i,k\in [n], j\in [m] \text{ with } k<i$.
See Figure~\ref{fig:K3,4} for an illustration of a completely interfered $K_{3,4}^1$.

\tikzset{every picture/.style={line width=0.75pt}} %set default line width to 0.75pt        

\begin{figure}
	\centering
	\begin{tikzpicture}[x=0.75pt,y=0.75pt,yscale=-0.5,xscale=0.6]

	%uncomment if require: \path (0,493); %set diagram left start at 0, and has height of 493
	
	%Straight Lines [id:da3586756295503366] 
	\draw [line width=1.01]    (110,50) -- (50,270) ;
	%Straight Lines [id:da005373931623811945] 
	\draw [line width=1.01]    (110,50) -- (90,270) ;
	%Straight Lines [id:da07603267290102744] 
	\draw [line width=1.01]    (110,50) -- (130,270) ;
	%Straight Lines [id:da18209758239973772] 
	\draw [line width=1.01]    (110,50) -- (170,270) ;
	%Straight Lines [id:da8179252057025777] 
	\draw [line width=1.01]    (350,50) -- (290,270) ;
	%Straight Lines [id:da7748049766079823] 
	\draw [line width=1.01]    (350,50) -- (330,270) ;
	%Straight Lines [id:da7552521034998396] 
	\draw [line width=1.01]    (350,50) -- (370,270) ;
	%Straight Lines [id:da19547309961336623] 
	\draw [line width=1.01]    (350,50) -- (410,270) ;
	%Straight Lines [id:da08282481432034428] 
	\draw [line width=1.01]    (590,50) -- (530,270) ;
	%Straight Lines [id:da419757759336183] 
	\draw [line width=1.01]    (590,50) -- (570,270) ;
	%Straight Lines [id:da8980271348920847] 
	\draw [line width=1.01]    (590,50) -- (610,270) ;
	%Straight Lines [id:da4951786921963419] 
	\draw [line width=1.01]    (110,50) -- (410,270) ;
	%Straight Lines [id:da39754450269567077] 
	\draw [line width=1.01]    (110,50) -- (370,270) ;
	%Straight Lines [id:da11740780112270621] 
	\draw [line width=1.01]    (110,50) -- (330,270) ;
	%Straight Lines [id:da36606718902044366] 
	\draw [line width=1.01]    (110,50) -- (290,270) ;
	%Straight Lines [id:da5899596012883572] 
	\draw [line width=1.01]    (110,50) -- (610,270) ;
	%Straight Lines [id:da8609085189344134] 
	\draw [line width=1.01]    (110,50) -- (650,270) ;
	%Straight Lines [id:da29245126328376614] 
	\draw [line width=1.01]    (110,50) -- (570,270) ;
	%Straight Lines [id:da23887284640676532] 
	\draw [line width=1.01]    (110,50) -- (530,270) ;
	%Straight Lines [id:da520632141591344] 
	\draw [line width=1.01]    (350,50) -- (530,270) ;
	%Straight Lines [id:da518478660301702] 
	\draw [line width=1.01]    (350,50) -- (570,270) ;
	%Straight Lines [id:da6453040413353224] 
	\draw [line width=1.01]    (350,50) -- (650,270) ;
	%Straight Lines [id:da26679890126501293] 
	\draw [line width=1.01]    (350,50) -- (610,270) ;
	%Straight Lines [id:da9459476790976042] 
	\draw [line width=1.01]    (50,270) -- (50,450) ;
	%Straight Lines [id:da812094383086116] 
	\draw [line width=1.01]    (290,270) -- (50,450) ;
	%Straight Lines [id:da8372375870129034] 
	\draw [line width=1.01]    (530,270) -- (50,450) ;
	%Straight Lines [id:da9829904893796744] 
	\draw [line width=1.01]    (90,270) -- (250,450) ;
	%Straight Lines [id:da4748922979395933] 
	\draw [line width=1.01]    (330,270) -- (250,450) ;
	%Straight Lines [id:da017325940703709186] 
	\draw [line width=1.01]    (570,270) -- (250,450) ;
	%Straight Lines [id:da7408319415315308] 
	\draw [line width=1.01]    (130,270) -- (394.49,418.78) -- (450,450) ;
	%Straight Lines [id:da18920785638307303] 
	\draw [line width=1.01]    (370,270) -- (450,450) ;
	%Straight Lines [id:da8231490102796695] 
	\draw [line width=1.01]    (610,270) -- (450,450) ;
	%Straight Lines [id:da3728828121152754] 
	\draw [line width=1.01]    (170,270) -- (650,450) ;
	%Straight Lines [id:da7521375777442618] 
	\draw [line width=1.01]    (650,450) -- (410,270) ;
	%Straight Lines [id:da3512632577131116] 
	\draw [line width=1.01]    (650,270) -- (650,450) ;
	%Shape: Circle [id:dp4051947588358067] 
	\draw  [fill={rgb, 255:red, 0; green, 0; blue, 0 }  ,fill opacity=1 ] (105,50) .. controls (105,47.24) and (107.24,45) .. (110,45) .. controls (112.76,45) and (115,47.24) .. (115,50) .. controls (115,52.76) and (112.76,55) .. (110,55) .. controls (107.24,55) and (105,52.76) .. (105,50) -- cycle ;
	%Shape: Circle [id:dp4473501801605966] 
	\draw  [fill={rgb, 255:red, 0; green, 0; blue, 0 }  ,fill opacity=1 ] (345,50) .. controls (345,47.24) and (347.24,45) .. (350,45) .. controls (352.76,45) and (355,47.24) .. (355,50) .. controls (355,52.76) and (352.76,55) .. (350,55) .. controls (347.24,55) and (345,52.76) .. (345,50) -- cycle ;
	%Shape: Circle [id:dp40609078984339675] 
	\draw  [fill={rgb, 255:red, 0; green, 0; blue, 0 }  ,fill opacity=1 ] (585,50) .. controls (585,47.24) and (587.24,45) .. (590,45) .. controls (592.76,45) and (595,47.24) .. (595,50) .. controls (595,52.76) and (592.76,55) .. (590,55) .. controls (587.24,55) and (585,52.76) .. (585,50) -- cycle ;
	%Shape: Circle [id:dp6292411928987156] 
	\draw  [fill={rgb, 255:red, 0; green, 0; blue, 0 }  ,fill opacity=1 ] (525,270) .. controls (525,267.24) and (527.24,265) .. (530,265) .. controls (532.76,265) and (535,267.24) .. (535,270) .. controls (535,272.76) and (532.76,275) .. (530,275) .. controls (527.24,275) and (525,272.76) .. (525,270) -- cycle ;
	%Shape: Circle [id:dp38394446888230216] 
	\draw  [fill={rgb, 255:red, 0; green, 0; blue, 0 }  ,fill opacity=1 ] (565,270) .. controls (565,267.24) and (567.24,265) .. (570,265) .. controls (572.76,265) and (575,267.24) .. (575,270) .. controls (575,272.76) and (572.76,275) .. (570,275) .. controls (567.24,275) and (565,272.76) .. (565,270) -- cycle ;
	%Shape: Circle [id:dp197159879445596] 
	\draw  [fill={rgb, 255:red, 0; green, 0; blue, 0 }  ,fill opacity=1 ] (605,270) .. controls (605,267.24) and (607.24,265) .. (610,265) .. controls (612.76,265) and (615,267.24) .. (615,270) .. controls (615,272.76) and (612.76,275) .. (610,275) .. controls (607.24,275) and (605,272.76) .. (605,270) -- cycle ;
	%Shape: Circle [id:dp14444881155496825] 
	\draw  [fill={rgb, 255:red, 0; green, 0; blue, 0 }  ,fill opacity=1 ] (645,270) .. controls (645,267.24) and (647.24,265) .. (650,265) .. controls (652.76,265) and (655,267.24) .. (655,270) .. controls (655,272.76) and (652.76,275) .. (650,275) .. controls (647.24,275) and (645,272.76) .. (645,270) -- cycle ;
	%Shape: Circle [id:dp7976515119396619] 
	\draw  [fill={rgb, 255:red, 0; green, 0; blue, 0 }  ,fill opacity=1 ] (285,270) .. controls (285,267.24) and (287.24,265) .. (290,265) .. controls (292.76,265) and (295,267.24) .. (295,270) .. controls (295,272.76) and (292.76,275) .. (290,275) .. controls (287.24,275) and (285,272.76) .. (285,270) -- cycle ;
	%Shape: Circle [id:dp9758490403641045] 
	\draw  [fill={rgb, 255:red, 0; green, 0; blue, 0 }  ,fill opacity=1 ] (325,270) .. controls (325,267.24) and (327.24,265) .. (330,265) .. controls (332.76,265) and (335,267.24) .. (335,270) .. controls (335,272.76) and (332.76,275) .. (330,275) .. controls (327.24,275) and (325,272.76) .. (325,270) -- cycle ;
	%Shape: Circle [id:dp04905548385637215] 
	\draw  [fill={rgb, 255:red, 0; green, 0; blue, 0 }  ,fill opacity=1 ] (365,270) .. controls (365,267.24) and (367.24,265) .. (370,265) .. controls (372.76,265) and (375,267.24) .. (375,270) .. controls (375,272.76) and (372.76,275) .. (370,275) .. controls (367.24,275) and (365,272.76) .. (365,270) -- cycle ;
	%Shape: Circle [id:dp43275935742522464] 
	\draw  [fill={rgb, 255:red, 0; green, 0; blue, 0 }  ,fill opacity=1 ] (405,270) .. controls (405,267.24) and (407.24,265) .. (410,265) .. controls (412.76,265) and (415,267.24) .. (415,270) .. controls (415,272.76) and (412.76,275) .. (410,275) .. controls (407.24,275) and (405,272.76) .. (405,270) -- cycle ;
	%Shape: Circle [id:dp5614453470287204] 
	\draw  [fill={rgb, 255:red, 0; green, 0; blue, 0 }  ,fill opacity=1 ] (165,270) .. controls (165,267.24) and (167.24,265) .. (170,265) .. controls (172.76,265) and (175,267.24) .. (175,270) .. controls (175,272.76) and (172.76,275) .. (170,275) .. controls (167.24,275) and (165,272.76) .. (165,270) -- cycle ;
	%Shape: Circle [id:dp16480002331733967] 
	\draw  [fill={rgb, 255:red, 0; green, 0; blue, 0 }  ,fill opacity=1 ] (125,270) .. controls (125,267.24) and (127.24,265) .. (130,265) .. controls (132.76,265) and (135,267.24) .. (135,270) .. controls (135,272.76) and (132.76,275) .. (130,275) .. controls (127.24,275) and (125,272.76) .. (125,270) -- cycle ;
	%Shape: Circle [id:dp8081778788152314] 
	\draw  [fill={rgb, 255:red, 0; green, 0; blue, 0 }  ,fill opacity=1 ] (85,270) .. controls (85,267.24) and (87.24,265) .. (90,265) .. controls (92.76,265) and (95,267.24) .. (95,270) .. controls (95,272.76) and (92.76,275) .. (90,275) .. controls (87.24,275) and (85,272.76) .. (85,270) -- cycle ;
	%Shape: Circle [id:dp4166415401225265] 
	\draw  [fill={rgb, 255:red, 0; green, 0; blue, 0 }  ,fill opacity=1 ] (45,270) .. controls (45,267.24) and (47.24,265) .. (50,265) .. controls (52.76,265) and (55,267.24) .. (55,270) .. controls (55,272.76) and (52.76,275) .. (50,275) .. controls (47.24,275) and (45,272.76) .. (45,270) -- cycle ;
	%Shape: Circle [id:dp9106016519087448] 
	\draw  [fill={rgb, 255:red, 0; green, 0; blue, 0 }  ,fill opacity=1 ] (45,450) .. controls (45,447.24) and (47.24,445) .. (50,445) .. controls (52.76,445) and (55,447.24) .. (55,450) .. controls (55,452.76) and (52.76,455) .. (50,455) .. controls (47.24,455) and (45,452.76) .. (45,450) -- cycle ;
	%Shape: Circle [id:dp7859772790720494] 
	\draw  [fill={rgb, 255:red, 0; green, 0; blue, 0 }  ,fill opacity=1 ] (245,450) .. controls (245,447.24) and (247.24,445) .. (250,445) .. controls (252.76,445) and (255,447.24) .. (255,450) .. controls (255,452.76) and (252.76,455) .. (250,455) .. controls (247.24,455) and (245,452.76) .. (245,450) -- cycle ;
	%Shape: Circle [id:dp1303519017352539] 
	\draw  [fill={rgb, 255:red, 0; green, 0; blue, 0 }  ,fill opacity=1 ] (445,450) .. controls (445,447.24) and (447.24,445) .. (450,445) .. controls (452.76,445) and (455,447.24) .. (455,450) .. controls (455,452.76) and (452.76,455) .. (450,455) .. controls (447.24,455) and (445,452.76) .. (445,450) -- cycle ;
	%Shape: Circle [id:dp46175117714727465] 
	\draw  [fill={rgb, 255:red, 0; green, 0; blue, 0 }  ,fill opacity=1 ] (645,450) .. controls (645,447.24) and (647.24,445) .. (650,445) .. controls (652.76,445) and (655,447.24) .. (655,450) .. controls (655,452.76) and (652.76,455) .. (650,455) .. controls (647.24,455) and (645,452.76) .. (645,450) -- cycle ;
	%Straight Lines [id:da21295635492339438] 
	\draw [line width=1.01]    (590,50) -- (650,270) ;

	\end{tikzpicture}
	\caption{A completely interfered $K_{3,4}^1$.}
	\label{fig:K3,4}
\end{figure}
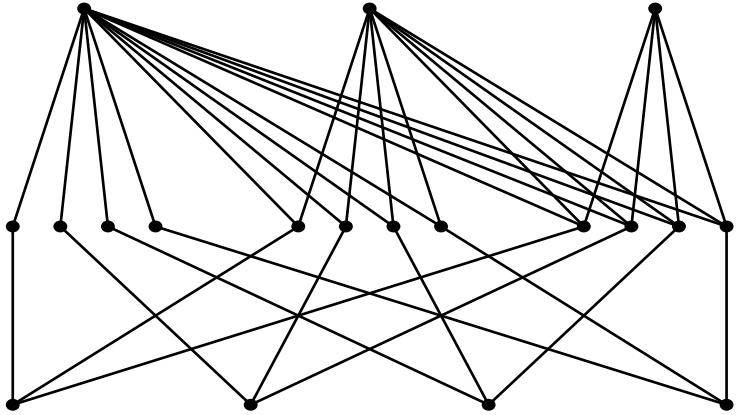

\begin{lemma}\label{Ramsey K^1_n,n}
	Let $G$ be an interfered $K_{N_1,N_2}^1$ where $N_1= R(n,n)$ and $N_2= 2^{N_1 \choose 2}n+1$. Then $G$ contains $K_{n,n}^1$ as a pivot-minor.
\end{lemma}

\begin{proof}
	For each $j\in [N_2]$, let $G_j$ be an auxiliary graph on the vertex set $[N_1]$ such that for each pair $k<i$, the vertex $k$ is adjacent to the vertex $i$ if and only if $x_k$ is adjacent to $z_{i,j}$ in $G$.
	By the pigeonhole principle there exists some $Y\subseteq [N_2]$ with $|Y|\ge n+1$ such that for each $j,j'\in Y$, we have that $G_j=G_{j'}$.
	Therefore for each $k<i$, the vertex $x_k$ of $G$ is either complete or anti-complete to $\{z_{i,j}: j\in Y\}$.
	
	Now consider another auxiliary graph $A$, this time on the vertex set $[N_1]$ such that for each pair $k<i$, the vertex $k$ is adjacent to the vertex $i$ if and only if $x_k$ is complete to $\{z_{i,j}: j\in Y\}$ in $G$. A stable set $I$ of $A$ corresponds to an induced $K_{|I|,|Y|}^1$, while a clique $C$ corresponds to a completely interfered $K_{|C|,|Y|}^1$. Hence we may assume that $G$ contains an induced subgraph $H$ that is a completely interfered $K_{n,n+1}^1$.
	
	Then observe that \[(H\wedge x_nz_{n,n+1} \wedge \dots \wedge x_1z_{1,n+1}) - \{y_{n+1},x_n,\dots ,x_1 \}\] is isomorphic to $K_{n,n}^1$ as we require.
\end{proof}

As a consequence of Lemma~\ref{K^1_n,n} and Lemma~\ref{Ramsey K^1_n,n} we get the following.

\begin{lemma}\label{Int K^1_n,n}
	For every graph $G$ there exist integers $q$ and $h$ such that every interfered $K_{q,h}^1$ contains $G$ as a vertex-minor.
\end{lemma}

So to find a graph $G$ as a vertex-minor it is enough to find an interfered $K_{q',h'}^1$ for some sufficiently large $q'$ and $h'$ that depend only on $G$. This is the tactic we follow in proving that vertex-minor-closed classes of graphs are 9-controlled.

One may naturally partition the edge set of a interfered $K_{q',h'}^1$ into two halfs, an ``non-interfered'' half that consists of $h$ disjoint stars and an ``interfered'' half that consists of $q$ stars with some possible additional edges between them.  Sections~3 and 5 deal with finding induced structures that shall contain as vertex-minors stars of the ``interfered'' and ``non-interfered'' half of the $K_{q,h}^1$ respectively. Sections~4 and 6 shall deal with showing that vertex-minors can simulate suitable edge-like contraction operations on each of these ``interfered'' and ``non-interfered'' induced structures respectively. Then Section~7 shall make use of these results and Lemma~\ref{Int K^1_n,n} to prove Theorem~\ref{liner 9-control}, that proper vertex-minor-closed classes of graphs are (linearly) 9-controlled.

Afterwards in Section~8 we make use of a theorem of Chudnovsky, Scott and Seymour~\cite{chudnovsky2016induced} to very quickly extend the work of Sections 3-7 by proving that proper vertex-minor-closed classes of graphs are 2-controlled.
In Section~9 we use this to prove Theorem~\ref{main} that proper vertex-minor-closed classes of graphs are $\chi$-bounded. Lastly, in Section~10, we extend the work from Section~9 a little further to prove Theorem~\ref{pivot step}. We remark that Sections~3-7, Section~8 and Sections~9-10 (each of which is dedicated to a different subproblem in the usual $\rho$-control strategy) can be read independently of each other.

\section{Large induced bloated trees}

This section is devoted to the analysis of large induced tree-like structures. From these structures we shall later obtained the ``interfered'' half of the $K_{q',h'}^1$ that we seek. Statements proven in this section are tailored to our needs but some of them could possibly be of more general interest.

If $T$ is a tree then we say that a vertex of degree at most $1$ is a \emph{leaf} and that a vertex of degree at least $3$ is a \emph{branching vertex}. The degree sum of an $n$-vertex tree is $2n-2$, therefore if a tree has $\ell>1$ leaves, then it has at most $\ell -2$ branching vertices. Likewise, a tree with $b \ge 1$ branching vertices has at least $b+2$ leaves. We use these two facts repeatedly without reference.

We call maximal cliques (with respect to vertex inclusion) of size at least three \emph{big cliques}, or \emph{big $k$-cliques} when we wish to refer to their size. We say that a graph $G$ is a \emph{bloated tree} if
\begin{itemize}
	\item every edge is contained in at most one big clique,
	\item the vertices of every big clique of size $k\ge3$ have degree at most $k$, and
	\item the graph obtained by contracting each big clique is a tree.
\end{itemize}

Note that a bloated tree is not necessarily a tree, but all trees are bloated trees.
An alternative definition for bloated trees is that they are block graphs such that for each $k\ge 3$, the vertices that are contained in the big $k$-cliques have degree at most $k$.
We say that a vertex of a bloated tree is a \emph{leaf} if it has degree at most 1. A vertex of a bloated tree is \emph{branching} if it has degree at least 3 and is not contained in a triangle.

Erd\H{o}s, Saks and S{\'o}s~\cite{erdos1986maximum} proved that for each $r\ge 3$, there exists an increasing function $t_r:\mathbb{N} \to \mathbb{N}$ such that $\lim_{n \to \infty}t_r(n) = \infty$, and every connected $K_r$-free graph with at least $n$ vertices contains an induced tree on at least $t_r(n)$ vertices. Fox, Loh and Sudakov~\cite{fox2009large} later improved the optimal bounds for $t_r$ up to constant factors.

We require a version for bloated trees, and for convenience we may make this independent of the clique number. We do not attempt to optimize the bounds. A clique is itself a bloated tree, so letting $f(n)=t_n(n)$ we obtain the following version of the theorem of Erd\H{o}s, Saks and S{\'o}s.

\begin{theorem}[Erd\H{o}s, Saks and S{\'o}s~\cite{erdos1986maximum}]\label{big bloated tree}
	There exists an increasing function $f:\mathbb{N}\to \mathbb{N}$ such that $\lim_{n \to \infty}f(n) = \infty$, and if $G$ is a connected graph on at least $n$ vertices, then it contains an induced bloated tree $T$ on at least $f(n)$ vertices.
\end{theorem}

We will further require a suitable version of Theorem~\ref{big bloated tree} in which we seek an induced bloated tree containing distinguished vertices. First we need a lemma on cut vertices and bridges in maximal induced bloated trees of a graph.
A vertex $v$ of a connected graph $G$ is a \emph{cut vertex} if $G-v$ is disconnected. Similarly, an edge $e$ of a connected graph $G$ is a \emph{bridge} if $G - e$ is disconnected.

\begin{lemma}\label{bridge}
	Let $G$ be a connected graph and $T$ a maximal induced bloated tree of $G$. If $u$ and $v$ are adjacent vertices that have degree two in $T$, and both $u$ and $v$ are cut vertices in $G$, then $uv$ is a bridge of $G$.
\end{lemma}

\begin{proof}
	As $u$ is a cut vertex of $G$ and has degree two in the maximal induced bloated tree $T$, we observe that, $G - u$ must have exactly two connected components. In particular the two vertices that $u$ is adjacent to in $T$ are in separate connected components of $G - u$. So $u$ is not contained in a big clique of $T$. Similarly for $v$.
	
	Suppose that $uv$ is not a bridge. Then there exists an induced cycle $C$ of $G$ containing the edge $uv$. No vertex of $C - \{u,v\}$ is adjacent to any vertex of $T - \{u,v\}$ as this would contradict the fact that both $G - u$ and $G - v$ have exactly two connected components. Consider the vertex $w$ of $C$ which is adjacent to $u$ and distinct from $v$. Then as $u$ and $v$ are cut vertices, we see that $u$ and possibly $v$ are the only vertices of $T$ that are adjacent to $w$ in $G$. This contradicts the maximality of the induced bloated tree $T$ as we may add the vertex $w$. We conclude that $uv$ is a bridge of $G$.
\end{proof}

\begin{theorem}\label{distinguished tree}
	There exists an increasing function $g:\mathbb{N} \to \mathbb{N}$ such that $\lim_{n \to \infty}g(n) = \infty$, and if $G$ is a connected graph and $S$ is a non-empty subset of its vertices, then $G$ contains an induced bloated tree with at least $g(|S|)$ vertices of $S$.
\end{theorem}

\begin{proof}
	Let $f$ be as in Theorem~\ref{big bloated tree}, and let $g(n)=\lceil\frac{1}{12}f(n)\rceil$. We will show that $g$ satisfies the conclusion of the theorem.
	
	Consider the graph $G'$ obtained from $G$ by repeating the following two operations until neither can be done.
	
	\begin{itemize}
		\item If $v$ is not a vertex of $S$ and not a cut vertex, then delete $v$.
		\item If $v$ is not a vertex of $S$ and $v$ has degree 2 in $G$ with both incident edges being bridges, then contract an edge incident to $v$.
	\end{itemize}

	Observe that any induced bloated tree of $G'$ corresponds to an induced bloated tree of $G$ that contains the same vertices from $S$. Hence we just need to find an induced bloated tree of $G'$ that contains at least $\frac{1}{12}f(|S|)$ vertices of $S$.
	By definition of $f$, we can find a maximal induced bloated tree $T'$ of $G'$ with at least $f(|V(G')|) \ge f(|S|)$ vertices. We will show that $|V(T')\cap S| \ge \frac{1}{12}|V(T')|$.
	
	Let:
	\begin{itemize}
	\item $\ell$ be equal to the number of leaves of $T'$,
	
	\item $b$ be equal to the number of branching vertices of $T'$,
	
	\item $x$ be equal to the number of big cliques $X$ of $T'$ having at least three vertices with a neighbour contained in $V(T') - X$,
	
	\item $y$ be equal to the number of big cliques $X$ of $T'$ having exactly two vertices with a neighbour contained in $V(T') - X$, and
	
	\item $z$ be equal to the number of big cliques $X$ of $T'$ having exactly one vertex with a neighbour contained in $V(T') - X$.
	\end{itemize}
	
	Every vertex $v$ of $G'$ that does not belong to $S$ is a cut vertex of $G'$.
	So by the maximality of $T'$, the set $S$ contains every leaf of $T'$ and every vertex contained in a big clique $X$ of $T'$ that has no neighbour in $V(T') - X$.
	If $T'$ is a clique then $V(T')\subseteq S$, we so may assume that $\ell +z\ge 2$.
	So the set $W$ of leaves, branching vertices and vertices contained in big cliques of $T'$ must contain at least $\ell + y +2z$ vertices of $S$.
	
	For each big clique $X$ of $T'$, let $P_X$ be a path on the vertices of $X$ that have a neighbour contained in $V(T') - X$.
	Now let $T''$ be the tree obtained from $T'$ by replacing each big clique $X$ with the path $P_X$.
	If $X$ is a big clique of $T'$ that has exactly one vertex with a neighbour contained in $V(T') - X$, then this vertex is a leaf of $T''$.
	Otherwise if $X$ is a big clique of $T'$ that has at least two vertices with a neighbour contained in $V(T') - X$, then two of these vertices have degree 2 in $T''$, while the others are all branching vertices of $T''$.
	Hence $T''$ has $\ell+z$ leaves and therefore at most $\ell+z -2$ branching vertices.
	It further follows that $T'$ has at most $3(\ell+z -b-2)$ vertices $v$ contained in a big clique $X$ of $T'$ such that $v$ and at least two other vertices of $X$ have a neighbour contained in $V(T') - X$.
	So the number of vertices that are not contained in $S$, but are contained in a big clique of $T'$ is at most $3(\ell + z -b -2)+2y+z$.
	Therefore $|W-S| \le b + 3(\ell + z -b -2) +2y +z = 3\ell +2y + 4z -2b -6$.
	
	Since $\ell + y +2z  \le |W\cap S|$, we get that $|W-S|\le 3|W\cap S|$.
	Hence $|W| \le 4|W\cap S|$.
	
	Now let $\mathcal{P}$ be the connected components of $T' - W$. Clearly every graph contained in $\mathcal{P}$ is a path. Also by considering the tree obtained from $T'$ by contracting each big clique, we observe that $b + x \le \ell + z - 2$. Therefore $|\mathcal{P}| \le \ell + b + x + y + z -1 \le 2\ell + y + 2z -3$.
	
	If $\sum_{P \in \mathcal{P}} |V(P)| \le 4(2\ell + y + 2z -3)$, then we have that \[|V(T')| \le 4|W\cap S| + 4(2\ell + y + 2z -3) <  12|W\cap S| \le 12|V(T') \cap S|.\]
	Hence we may assume that $\sum_{P \in \mathcal{P}} |V(P)| > 4(2\ell + y + 2z -3)$.
	
	Consider one such path $P\in \mathcal{P}$.
	Suppose that $P$ contains three consecutive vertices $u,v,w$ that are all not contained in $S$. Then $u,v,w$ would all be cut vertices of $G'$. 
	By Lemma~\ref{bridge}, both $uv$ and $vw$ must be bridges of $T'$. Hence by the maximality of $T'$, $v$ has degree 2 in $G'$. But now this contradicts the choice of $G'$, so we may conclude that $P$ contains no three consecutive vertices that are all not contained in $S$.
	Hence $|V(P) \cap S| \ge \left\lfloor \frac{|V(P)|}{3} \right\rfloor$.
	
	Then summing over all $P\in \mathcal{P}$, we get that
	\begin{align*}
	\sum_{P \in \mathcal{P}} |V(P)\cap S|
	&\ge \sum_{P \in \mathcal{P}} \frac{  |V(P)| }{3} - (2\ell + y + 2z -3) \\
	&> \sum_{P \in \mathcal{P}} \frac{  |V(P)| }{3} - \sum_{P \in \mathcal{P}} \frac{  |V(P)| }{4} \\
	&=  \sum_{P \in \mathcal{P}} \frac{  |V(P)| }{12}.
	\end{align*}
	So it follows that \[|V(T') \cap S|=  |W\cap S| + \sum_{P \in \mathcal{P}} |V(P)\cap S| >  \frac{|W|}{4} + \sum_{P \in \mathcal{P}}  \frac{  |V(P)| }{12} \ge         \frac{|V(T')|}{12}\] as desired.
\end{proof}

Next we aim to prune bloated trees with many leaves to obtain a smaller bloated tree, still with many leaves, but without near branching vertices or big cliques.

The following lemma is due to Esperet and de Joannis de Verclos~\cite{branchingEsperet}.

\begin{lemma}[Esperet and de Joannis de Verclos~\cite{branchingEsperet}]\label{branching}
	Every tree $T$ with at least $\ell$ leaves has a subtree which contains at least $\sqrt{\ell}$ of the leaves of $T$ and has no adjacent branching vertices.
\end{lemma}

\begin{proof}
	We will prove a stronger statement on rooted trees. For a rooted tree $T$, let $f_0(T)$ be the largest number of leaves of $T$ in a subtree of $T$ that includes the root vertex and all its children without having adjacent branching vertices.
	Similarly for a rooted tree $T$, let $f_1(T)$ be equal to the largest number of leaves of $T$ in a subtree of $T$ that contains the root vertex and at most one of its children without having adjacent branching vertices.
	
	We will prove that $f_0(T)\cdot f_1(T) \ge \ell$, which clearly implies the lemma.
	If $T$ has height either 0 or 1 then the result is clear. So we may assume that $T$ has height at least 2. We proceed by induction on the height of $T$. Let $T_1, \dots ,T_k$ be the subtrees obtained by taking a child of the root of $T$, rooting at this vertex and then taking all its descendants.
	Then $f_0(T) \ge  \sum_{i=1}^{k}f_1(T_i)$ and $f_1(T) \ge \max\{f_0(T_i): i\in \{1,\dots ,k\}\}$.
	Hence \[f_0(T)\cdot f_1(T) \ge \sum_{i=1}^{k} (f_1(T_i) \cdot f_0(T_i)) \ge \ell.\] as required.
\end{proof}

\begin{lemma}\label{bloated branching}
	Let $T$ be a bloated tree with $\ell$ leaves. Then $T$ contains an induced bloated tree $T'$ that has at least $\ell^{\frac{1}{4}}$ leaves and whose branching vertices and big cliques are all at distance at least 4 from every other branching vertex or big clique of $T'$.
\end{lemma}

\begin{proof}
	First note that if a vertex of a  big clique of size $k$ has degree $k-1$, then we may just delete the vertex. So we may assume that $T$ has no such vertex.
	
	Next we reduce the problem to trees. We may contract each big clique of $T$ into a single vertex to obtain a tree $T^*$. Now by reversing this operation we observe that such a desired subtree of $T^*$ corresponds to such a desired induced subgraph of $T$. Hence we may assume that $T$ is a tree. We need only consider the distance between two branching vertices in $T$ as there are no big cliques in a tree.
	
	By Lemma~\ref{branching}, the tree $T$ contains an induced subtree $T_2$ with at least $\sqrt{\ell}$ leaves and with no adjacent branching vertices. Now by considering the graph obtained by smoothing degree-2 vertices of $T_2$ and applying Lemma~\ref{branching} again, we may find a subtree $T'$ of $T_2$ and so of $T$ with at least $\ell^{\frac{1}{4}}$ leaves and with no pair of branching vertices at a distance of less than 4 from each other.
\end{proof}

Next we combine the previous few lemmas so that we may find our desired bloated trees, this is the main result of this section.

\begin{lemma}\label{1 L}
	For every positive integer $\ell$, there exists a positive integer $\ell'$ such that every connected graph $G$ with a set $S$ of $\ell'$ distinguished vertices of degree 1 contains an induced bloated tree $T$ whose branching vertices and big cliques are all at distance at least 4 from each other and with $\ell$ leaves, all contained in $S$.
\end{lemma}

\begin{proof}
	By Theorem~\ref{distinguished tree}, there exists some positive integer $\ell'$ such that every connected graph $G$ with a set $S$ of $\ell'$ distinguished vertices of degree 1 contains an induced bloated tree $T'$ with $\ell^4$ leaves, all contained in $S$. Now by Lemma~\ref{bloated branching}, there is an induced subgraph $T$ of $T'$ that is a bloated tree with $\ell$ leaves all contained in $S$ and whose big cliques and branching vertices are all at distance at least 4 from each other as required.
\end{proof}

\section{Vertex-minors and induced bloated trees}

In this section we will be concerned with simulating an edge contraction-like operation on bloated trees by using vertex-minors. We will require some additional properties on how the induced bloated tree $T$ interacts with the rest of the graph it lies in than was obtained in Section~3, however this shall be dealt with in Section~7.

The next lemma will allow us to eliminate the big cliques from these bloated trees.

\begin{lemma}\label{elminate big clique}
	Let $c$ be a degree-$k$ vertex of a graph $G$ contained in a big $k$-clique $C$ such that it has a single neighbour $d$ that is not adjacent to any other vertex of $C$. Then $(G-E(C-c))/cd$ is a vertex-minor of $G$.
\end{lemma}

\begin{proof}
	Simply consider $G*c-c$.
\end{proof}

The next lemma will allow us to eliminate undesirable ``fanning'' that interferes with the bloated tree.

\begin{lemma}\label{elminate fanning}
	Let $H$ be an induced subgraph of a graph $G$ such that $H$ consists of an induced path $av_0v_1\dots v_{k}b$ and an additional vertex $c$ with no neighbours in $\{a,v_0,b\}$ and such that $N_G(\{v_0,v_1,\dots v_{k}\})  \subseteq \{a,b,c\}$. Then $G$ contains either $G/\{v_0v_1,\dots, v_{k-1}v_{k}\}$ or  $(G-E(c,\{v_1,\dots ,v_k\}))/\{v_0v_1,\dots, v_{k-1}v_{k}\}$ as a vertex-minor.
\end{lemma}

\begin{proof}
	By smoothing vertices we may assume that $\{v_1, \dots , v_k\} \subseteq N(c)$. The $k=0$ case is trivially true. If $k=1$ then we may obtain the desired vertex-minor by smoothing $v_0$. If $k=2$ then: \[(G * v_1 *v_2)-v_1-v_2 = G/\{v_0v_1, v_1v_2\}.\]
	If $k=3$, then: \[(G*v_{2}*v_{1}*v_3)-v_{1}-v_{2}-v_3= (G-E(c,\{v_1,v_2,v_3\}))/\{v_0v_1,v_1v_2, v_{2}v_{3}\}.\]
	Similarly for $k> 3$ we may reduce to the $k-3$ case via the vertex-minor; \[{(G*v_{k-1}*v_{k-2}*v_k)}-v_{k-2}-v_{k-1}-v_k=G/\{v_{k-3}v_{k-2},v_{k-2}v_{k-1},v_{k-1}v_k\}.\]
	
	This completes the proof.
\end{proof}

Consider an induced bloated tree $T$ of a graph $G$. Let $L$ be the set of leaves of $T$. Let $B$ be the set of branching vertices. Let $Z$ be the set of vertices $z\in V(T)-L$ such that $N_G(z)-V(T)$ is non-empty.

We call $T$ \emph{shrinkable} if

\begin{itemize}
	\item for each $z\in Z$, $|N_G(z)-V(T)|=1$,
	
	\item in $T$, the distance between pairs of vertices $z,z'\in Z$ with $N_G(z)-V(T) \not= N_G(z')-V(T)$ is at least 4,
	
	\item in $T$, no vertex of $B\cup Z$ is within distance 3 of a big clique,
	
	\item in $T$, no vertex of $Z$ is within distance 3 of a vertex in $B$, and
	
	\item in $T$, vertices of $L$ are at distance at least 2 from every big clique and distance at least 3 from every vertex of $Z$.
\end{itemize}

In a similar but much simpler manner, we say that an induced tree $T$ is \emph{shrinking} if $B\cup Z$ is a stable set. We remark that the sets $B$ and $Z$ need not be disjoint in the definition of a shrinking tree. The next step is to modify shrinkable bloated trees into shrinking trees.

\begin{lemma}\label{shrinkable to shrinking}
	Let $T$ be an induced shrinkable bloated tree of a graph $G$ and let $L$ be the set of leaves of $T$.
	Then there is a vertex-minor $G'$ of $G$ with an induced shrinking tree $T'$ that has $L$ as its set of leaves such that $G'-(V(T')-L)= G-(V(T)-L)$ and $N_{G'}(V(T')-L ) \subseteq N_{G}(V(T)-L )$.
\end{lemma}

\begin{proof}
	Firstly by appropriately removing vertices of $T-L$ we can assume that no big clique of $T$ contains a vertex with no neighbour in $T$ that is outside the big clique.
	
	Now we apply Lemma~\ref{elminate big clique} to a vertex of each big clique of
	$T$ to obtain a vertex-minor $G^*$ of $G$ and an induced tree $T^*$ of $G^*$ with its set of leaves being $L$, its set $B^*$ of branching vertices, and $Z^*={\{ z\in V(T)-L : N_{G^*}(z)-V(T^*) \not= \emptyset \}}$ such that:
	\begin{itemize}
		\item for each $z\in Z^*$, $|N_{G^*}(z)-V(T^*)|=1$,
		
		\item in $T^*$, the distance between pairs of vertices $z,z'\in Z$ with $N_{G^*}(z)-V(T^*) \not= N_{G^*}(z')-V(T^*)$ is at least 4,
		
		\item in $T^*$, all vertices of $B^*$ are at distance at least 3 from vertices of $Z^*$ and at distance at least 2 from other vertices of $B^*$, and
		
		\item in $T^*$, vertices of $L$ are at distance at least 3 from vertices of $Z^* $.
	\end{itemize}

	In particular, this eliminated the big cliques of $T$. Next we must eliminate the undesirable ``fanning''. Now we may partition the vertices of $T^*-B^*-L$ into sets $X_1,\dots , X_h$ such that for each $i\in [h]$, we have $|N_{G^*}(X_i)-V(T^*)| \le 1$, and $G^*[X_i]$ is a path with endpoints $a$ and $b$ and a vertex $v_0$ adjacent to $a$ with $\{a,v_0,b\}\not\in Z^*$. Then to obtain our desired vertex-minor $G'$ we just apply Lemma~\ref{elminate fanning} to each set of the induced paths $G^*[X_i]$ such that $X_i\cap Z^*\not= \emptyset$.
\end{proof}

With this we may now simulate another contraction operation on shrinkable bloated trees.

\begin{lemma}\label{contract bloated tree}
Let $T$ be an induced shrinkable bloated tree of a graph $G$, with leaves $L$.
Then there is a vertex-minor $G'$ of $G$ with an induced star $T'$ that has $L$ as its set of leaves such that $G'-(V(T')-L)= G-(V(T)-L)$ and $N_{G'}(V(T')-L ) \subseteq N_{G}(V(T)-L )$.
Or equivalently there exists some subset $E^*$ of $E(V(T)-L,V(G)-V(T))$ such that $G'=(G-E^*)/E(T-L)$ is a vertex-minor of $G$.
\end{lemma}

\begin{proof}
Firstly by Lemma~\ref{shrinkable to shrinking}, we may instead assume that $T$ is a shrinking tree of $G$.

If $|V(T)-L|=1$ then the result is trivial. If $|V(T)-L|=2$ then we may simply smooth some vertex. Suppose for sake of contradiction that $T$ is a counter-example with $|V(T)-L|$ minimum, we may assume that $|V(T)-L|\ge 3$.

Suppose first that there exist two adjacent vertices $u$ and $v$ of $T$ such that $u,v \not\in L\cup B\cup Z$. Then both $u$ and $v$ have degree 2 in $G$ and we may smooth one, but this contradicts $|V(T)-L|$ being minimum.

Similarly suppose that there exists a vertex $v$ of $T$ such that $v \not\in L\cup B\cup Z$ and $v$ is adjacent to a leaf of $T$. Then we may again smooth $v$ contradicting $|V(T)-L|$ being minimum.

So there must exist a vertex $v\in V(T)-(L\cup B\cup Z)$ of degree 2 with neighbours $u$ and $w$ such that $u,w \in (B\cup Z)-L$. Let $E'$ be the set of edges between either $u$ or $w$ and the set $(N_G(u)\cap N_G(w))-v \subseteq N_{G}(V(T)-L )$. Then $(G \wedge uv) -u-v$ is the graph obtained by deleting the set of edges $E'$, and then contracting edges $uv$ and $vw$. This again contradicts $|V(T)-L|$ being minimum and so completes the proof.
\end{proof}

\section{Distant paths}

In this section we show how to find distant paths within a set of high chromatic number such that each path also contains distant vertices. These paths will be used to build the ``non-interfered'' half of a large interfered $K_{q',h'}^1$.
To set the mood for the next lemma, one may view the argument as a variation on the classical Gy\'{a}rf\'{a}s path argument~\cite{gyarfas1985problems}.

A \emph{lollipop} in a graph $G$ is a pair $(P,C)$ where $C\subseteq V(G)$, $G[C]$ is connected, and $P$ is an induced path of $G$ with an end vertex $t$ such that $t$ has a neighbour in $C$, and is the only vertex of $P$ with a neighbour in $C$.
A $q$-stripe of a lollipop $(P,C)$ is a set of $q$ vertices $\{s_1,\dots , s_q\}\subseteq V(P)$ such that $s_1,\dots, s_q, C$ are pairwise at distance at least 8 from each other in $G$.
A lollipop $(P,C)$ is \emph{contained} in a set $X$ of vertices if $V(P),C \subseteq X$.

\begin{lemma}\label{lollipop}
	Let $c,k,\kappa$ be positive integers. 
	Let $G$ be a graph such that $\chi^{(8)}(G) \le \kappa$ and let $C\subseteq V(G)$ be such that $\chi(C)\ge c+k\kappa$. Then there is a lollipop $(P,C')$ contained in $C$ with $\chi(C')\ge c$ and a $k$-stripe of $(P,C')$.
\end{lemma}

\begin{proof}
	First we handle the case that $k=1$. Let $s_1$ be a vertex of $C$, and let $C'$ be the vertices of a connected component of $G[C-N_7[s_1]]$ with chromatic number at least $\chi(C)-\kappa \ge c$. Then let $P$ be a shortest path in $G[C]$ between the vertex $s_1$ and $N(C')\cap C$. Let $t$ be the other end vertex of $P$. Then $(P,C')$ provides the desired lollipop with $\{s_1\}$ being a 1-stripe of $(P,C')$.
	
	So now we may proceed inductively. Let $(P^*,C^*)$ be a lollipop with $k-1$ stripes contained in $C$ with $\chi(C^*) \ge \chi(G)- (k-1)\kappa \ge c + \kappa$. Let $t^*$ be the end vertex of $P^*$ neighbouring a vertex of $C^*$.
	Let $\{s_1,\dots, s_{k-1}\}$ be a $(k-1)$-stripe of $(P^*,C^*)$. Let $C'$ be the vertex set of a connected component of $G[C^*-N_8[t^*]]$ with chromatic number at least $\chi(C^*) -\kappa \ge c$. Let $P'$ be a shortest path in $G[C^*\cup\{t^*\}]$ between $t^*$ and $N(C')\cap C^*$ and let $t$ be the other end vertex of $P'$. Let $P=P^*\cup P'$ and let $s_k$ be the vertex of $P$ adjacent to $t^*$ and contained in $C^*$. Then $(P,C')$ is a lollipop contained in $C$, and $\{s_1,\dots,s_k\}$ is a $k$-stripe of $(P,C')$ as required.
\end{proof}

The purpose of the lollipop structure was to aid us in finding the paths we seek. We may now be more precise with what we need from this section.

\begin{lemma}\label{distant paths}
	Let $q,h,\kappa$ be non-negative integers with $\kappa \ge 1$.
	Let $G$ be a graph such that $\chi^{(9)}(G) \le \kappa$ and let $C\subseteq V(G)$ be such that $\chi(C) \ge qh\kappa$. Then there exist induced paths $P_1,\dots , P_h$  contained in $C$ that are pairwise at distance at least 3 from each other in $G$ and such that, for each $j\in [h]$, the path $P_j$ contains a set $\{s_{1,j},\dots,s_{q,j}\}$ of $q$ vertices and the vertices of $\{s_{i,j} : i\in [q], j\in [h]\}$ are pairwise at distance at least 8 from each other.
\end{lemma}

\begin{proof}
	The result is vacuously true if $h=0$, so may proceed inductively.
	By Lemma~\ref{lollipop} there exist a lollipop $(P,C')$ contained in $C$ and a $q$-stripe of $(P,C')$. We remark that no condition on the chromatic number of $G[C']$ is necessary here and that we may take $C'=\emptyset$.
	
	Next we choose a set $\{s_{1,h}, \dots ,s_{q,h}\}$ of vertices contained in $V(P)$ that are pairwise at distance at least 8 from each other in $G$, and subject to this so that the distance in $P$ between the first and last vertex of $\{s_{1,h}, \dots ,s_{q,h}\}$ as they appear on $P$ is minimised.
	Such a set exists since there is a $q$-stripe of $(P,C')$. Let $P_h$ be the subpath of $P$ between the first and last vertex of $\{s_{1,h}, \dots ,s_{q,h}\}$ as they appear on $P$. Then by the choice of $\{s_{1,h}, \dots ,s_{q,h}\}$, we have that $\{s_{1,h}, \dots ,s_{q,h}\} \subseteq V(P_h)$ and ${V(P_h)\subseteq N_7[\{s_{1,h}, \dots ,s_{q,h}\}]}$.
	
	Now let $C^*$ be the vertex set of a connected component of the induced subgraph $G[C-N_9[\{s_{1,h}, \dots ,s_{q,h}\}]]$ such that $\chi(C^*) \ge {\chi(C)-q\kappa\ge q(h-1)\kappa}$. Notice that $C^*$ is at distance at least 3 from $V(P_h)$ and at distance at least 10 from $\{s_{1,h}, \dots ,s_{q,h}\}$ in $G$.
	
	Then by the inductive hypothesis we may find paths $P_1,\dots, P_{h-1}$ contained $C^*$ satisfying the conclusion of the lemma. Then $P_1,\dots , P_h$ provide the desired paths contained in $C$.
\end{proof}

\section{Vertex-minors and dangling paths}

In this section we show how to simulate another edge-contraction operation, this time on distant paths that interact with the rest of our graph in a particular way. This shall be used in finding the ``non-interfered'' stars of a large interfered~$K_{q',h'}^1$.

Let $G$ be a graph, we say that an induced path $P$ \emph{dangles} from a set ${X\subset V(G)}$ if $N(V(P))=X$.
If for each $x\in X$, there is an odd number of edges between $x$ and $V(P)$, then we say that the path \emph{dangles oddly}.

\begin{lemma}\label{hanging oddly}
	Let $P$ be a path dangling from a stable set $X$ in a graph $G$ and let $X'$ be the subset
	of $X$ consisting of all the vertices that have an even number of neighbours in $V(P)$. If the distance between two vertices in $X'$ is at least 3, then there is a vertex-minor $H$ of $G$ such that $G[V(G)-V(P)]= H[V(G)-V(P)]$ and $H[V(H)-(V(G)-V(P))]$ is an induced path $P'$ that dangles oddly from the set $X$.
\end{lemma}

\begin{proof}
	This is trivially true if $|X'|=0$, we so proceed inductively.
	If $X'$ is non-empty, then the path $P$ has at least two vertices.
	Let $x\in X'$, and let $y$ be a vertex of $V(P)$ that is adjacent to $x$.
	
	Suppose first that $y$ is an end vertex of $P$. Then let $G''=G -y$, and let $X''=X' - x$. Then in the graph $G''$, the path $P''=P-y$ dangles from the set $X$, and $X''$ is exactly the subset of $X$ consisting of the vertices with an even number of neighbours in $V(P'')$. Clearly in $G''= G -y$ the distance between vertices in $X''$ is at least 3 and $X$ remains a stable set. So by induction $G''$ (and thus $G$) contains a vertex-minor $H$ as in the conclusion of the lemma.
	
	So we may assume now in the second case that $y$ is not an end vertex of $P$.
	Let $a,b$ be the two vertices of $P$ adjacent to $y$.
	Then let $G'' = G*y-y$, and let $X''= X'-x$. The graph $G''$ is identical to $G$, except that $y$ is deleted, the edge $ab$ is added and the number of edges between $x$ and $\{a,b\}$ is still equal modulo 2. In particular $X$ is a stable set in $G''$.
	Also $P''=G''[V(P)-y]$ is an induced path dangling from $X$ in $G$ so that $X''$ is exactly the subset of $X$ consisting of the vertices with an even number of neighbours in $V(P'')$. The neighbourhood a vertex in $X-x$ is the same in both $G$ and $G''$. So in $G''$ the distance between vertices of $X''$ is at least 3.
	Then as in the first case, by induction we get that $G$ contains a vertex-minor $H$ as in the conclusion of the lemma.
\end{proof}

Next we show that vertex-minors can be used to simulate another edge contraction-like operation, this time on paths dangling oddly.

\begin{lemma}\label{ramsey hanging}
	Let $q$ be a positive integer and let $P$ be a path dangling oddly from a set $X$ of vertices in a graph $G$ with $|X|\ge R(q,q)$. Then there exists $Y\subseteq X$ with $|Y|\ge q$ such that the graph $(G-(X-Y)-E(Y))/E(P)$ is a vertex-minor of $G$.
\end{lemma}

\begin{proof}
	First we suppose that the path $P$ consists of a single vertex $p$. If there exists a stable set of $I$ in $X$ of size $q$, then it is enough to take $Y=I$, so we may assume not. Then by Ramsey theorem there exists a clique $C$ in $X$ of size $q$. In this case we may take $Y=C$ and $(G-(X-Y))*p$.
	
	We shall now proceed inductively.
	Let the path $P$ be $p_1p_2\dots p_m$, we may now assume that $m\ge 2$. Observe that in the graph $G'=G*p_m-p_m$, the path $P'=p_1p_2\dots p_{m-1}$ dangles oddly from the set $X$.
	Furthermore $G' - V(P')$ and $G- V(P)$ may differ only on the adjacencies between vertices of $X$.
	
	So by induction there exists $Y\subseteq X$ with $|Y|\ge q$ such that the graph $(G'-(X-Y)-E_{G'}(Y))/E(P') = (G-(X-Y)-E_G(Y))/E(P)$ is a vertex-minor of $G'$, and thus of $G$ as required.
\end{proof}

In a graph $G$, we say that a path dangling from a set $X$ \emph{dangles spaciously} if the distance between vertices in $X$ is at least 6.
By the previous two lemmas we obtain the main result of this section.

\begin{lemma}\label{contract 1 hanging path}
	Let $q$ be a positive integer and let $P$ be a path dangling spaciously from a vertex set $X$ with $|X|\ge R(q,q)$ in a graph $G$. Then there exists a $Y\subseteq X$ with $|Y|\ge q$ such that the graph $(G-(X-Y))/E(P)$ is a vertex-minor of $G$.
\end{lemma}

Of course this lemma still holds if we just required that vertices of $X$ are at distance at least 3 from each other, rather than at least 6 as in the definition of a path dangling spaciously. But for our purposes a minimum distance of at least 6 is enough.

\section{Linear 9-control}

In this section we shall see the fruits of our labour in Sections~3-6 and prove that vertex-minor-closed classes of graphs are linearly 9-controlled. We remark that with some more care one may certainly argue directly that such graphs are linearly $\rho$-controlled for some slightly smaller $\rho$. However showing 9-control shall be sufficient for later extending to 2-control in Section~8.

We call a collection ${\mathcal{L}} = (L^0,L^1,L^2,L^3)$ a \emph{long cover} of a set $C\subset V(G)$ if:

\begin{itemize}
	\item the subsets $L^0,L^1,L^2,L^3, C\subset V(G)$ are pairwise disjoint,
	
	\item $G[L^0]$ is connected,
	
	\item for each $i\in \{0,1,2\}$, $L^i$ dominates $L^{i+1}$, and $L^3$ dominates $C$,
	
	\item for each $i\in \{0,1,2\}$, $C$ is anti-complete to $L^i$, and
	
	\item for each $i,j\in \{0,1,2,3\}$, $L_i$ is anti-complete to $L_j$ if $|i-j|>1$.
\end{itemize}

We say that two long covers ${\mathcal{L}}_i = (L^0_i,L^1_i,L^2_i,L^3_i)$ and ${\mathcal{L}}_j = (L^0_j,L^1_j,L^2_j,L^3_j)$ are disjoint if the two sets $L^0_i \cup L^1_i \cup L^2_i \cup L^3_i$ and $L^0_j \cup L^1_j \cup L^2_j \cup L^3_j$ are disjoint.
We say that a collection of pairwise disjoint long covers $({\mathcal{L}}_i : i\in [q])$ of a set $C$ is a \emph{long $q$-cover} of a set $C\subset V(G)$ if for each $i,j\in [q]$, with $i<j$, the set of vertices $L^0_j\cup L^1_j\cup L^2_j$ is anti-complete to $L^0_i \cup L^1_i\cup L^2_i \cup L^3_i$.

We start by showing that for large $q$, we may find long $q$-covers of sets with large chromatic number. This is just a straightforward levelling argument.

\begin{lemma}\label{q-cover}
	Let $q, c, \kappa$ be non-negative integers with $\kappa \ge 1$. Then every graph $G$ satisfying \allowbreak ${\chi(G) > 2^q\max\{c,\kappa\}}$ and $\chi^{(3)}(G)\le \kappa$ contains a long $q$-cover \allowbreak ${({\mathcal{L}}_i : i\in [q])}$ of a set $C$, with $\chi(C)> c$.
\end{lemma}

\begin{proof}
	For $q=0$, the result is trivial. We proceed inductively on $q$. Let $G$ be such a graph and let $v$ be a vertex of $G$ in a component with largest chromatic number.
	Let $t$ be the smallest positive integer such that $G[N_t(v)]$ has chromatic number more than $2^{q-1}\max\{c,\kappa\}$ (such a $t$ exists as otherwise for each $i$, dependent on if $i$ is odd or even, we could colour the vertices of $N_i(v)$ from one of two sets of $2^{q-1}\max\{c,\kappa\}$ colours, thus yielding a colouring of $G$ with at most $2^q\max\{c,\kappa\}$ colours).
	Note that $t\ge 4$ as $\chi(N_{3}[v])\le \kappa$.
	Now $G[N_t(v)]$ contains a long $(q-1)$-cover $({\mathcal{L}}_i : i\in [q-1])$ of a set $C\subset N_t(v)$, with $\chi(C) > c$. Now let $L_q^0 = N_{t-4}[v]$, $L_q^1=N_{t-3}(v)$, $L_q^2=N_{t-2}(v)$, $L_q^3=N_{t-1}(v)$, and let ${\mathcal{L}}_q = (L_q^0,L_q^1,L_q^2,L_q^3)$. Then $({\mathcal{L}}_i : i\in [q])$ is a long $q$-cover of $C$ as required.
\end{proof}

We plan to find a large interfered $K_{q',h'}^1$ as a vertex-minor. We shall find the ``interfered'' half of our $K_{q',h'}^1$ within the vertices of some long $q$-cover for sufficiently large $q$. The stars of the ``non-interfered'' half of our $K_{q',h'}^1$ shall be found within certain paths disjoint from the long $q$-cover. For now we must focus on the ``non-interfered'' half of our $K_{q',h'}^1$.

Let $G$ be a graph containing a long $q$-cover $({\mathcal{L}}_i : i\in [q])$. We say that a collection of induced paths $P_1,\dots ,P_h$ that are disjoint from $({\mathcal{L}}_i : i\in [q])$ \emph{dangle spaciously} from the long $q$-cover $({\mathcal{L}}_i : i\in [q])$ if they are pairwise disjoint and anti-complete to each other, and there exists a set $M=\{m_{i,j} : i\in [q], j\in[h] \}$ of vertices such that

\begin{itemize}
	\item the vertices of $M$ are at distance at least 6 from each other,
	
	\item for each $i\in [q]$, the vertices $\{m_{i,1}, \dots , m_{i,h}\}$ are contained in $L_i^3$, and
	
	\item for each $j\in [h]$, the path $P_j$ dangles spaciously from the set $\{m_{1,j}, \dots , m_{q,j}\}$ of vertices in the subgraph of $G$ induced by the vertices of the long $q$-cover $({\mathcal{L}}_i : i\in [q])$ and the paths $P_1,\dots ,P_h$.
\end{itemize}

Next we use the main result of Section~5 to find such paths dangling spaciously from a long $q$-cover.

\begin{lemma}\label{long q-cover h hanging paths}
	Let $q,h, \kappa$ be positive integers and let $G$ be a graph such that $\chi^{(9)}(G) \le \kappa$ and $\chi(G) > 2^{q}qh\kappa$. Then $G$ contains a long $q$-cover $({\mathcal{L}}_i : i\in [q])$ with $h$ paths $P_1,\dots, P_h$ dangling spaciously.
\end{lemma}

\begin{proof}
	By Lemma~\ref{q-cover} there exists a set $C^*\subseteq V(G)$ with $\chi(C^*) > qh\kappa$ and a long $q$-cover $({\mathcal{L}}_i^* : i\in [q])= ((L_i^{*0}, L_i^{*1}, L_i^{*2}, L_i^{*3}) : i\in [q])$ of $C^*$. Now by Lemma~\ref{distant paths} there exist induced paths $P_1,\dots , P_h$ contained in $C^*$ that are pairwise at distance at least 3 from each other in $G$ and such that for each $j\in [h]$, the path $P_j$ contains a set of $q$ vertices $\{s_{1,j},\dots,s_{q,j}\}$ such that the vertices of $\{s_{i,j} : i\in [q], j\in [h]\}$ are pairwise at distance at least 8 from each other.
	
	Now for each $i\in [q], j\in [h]$, let $m_{i,j}$ be a vertex of $L_i^{*3}$ that is adjacent to $s_{i,j}$. Let $M=\{m_{i,j} : i\in [q], j\in[h] \}$. Then as the vertices of $\{s_{i,j} : i\in [q], j\in [h]\}$ are pairwise at distance at least 8 from each other, we see that the vertices of $M$ must be pairwise at distance at least 6 from each other. Then let $({\mathcal{L}}_i : i\in [q])$ be the long $q$-cover (of an empty set) obtained from the long $q$-cover $({\mathcal{L}}_i^* : i\in [q])$ by removing the vertices of $N(V(P_1) \cup \dots \cup V(P_h)) - M$ that are contained in $({\mathcal{L}}_i^* : i\in [q])$.
\end{proof}

Notice that in Lemma~\ref{long q-cover h hanging paths}, it is not important what set $C$ of vertices that $({\mathcal{L}}_i : i\in [q])$ is a long $q$-cover of, and indeed we may as well assume that the set $C$ is empty.
The next step will be to apply the main result of Section~6 to simulate an edge contraction-like operation on these paths dangling spaciously.

A graph $F$ is a \emph{$(q,h)$-frame} if there exists a partition of the vertices into sets $A_1,\dots, A_q , M=\{m_{i,j}: i\in [q], j\in [h]\}, S=\{s_1,\dots , s_h\}$ such that
\begin{itemize}
	\item for each $i\in[q]$, $F[A_i]$ is connected,
	
	\item the vertex sets $A_1,\dots , A_q$ are pairwise anti-complete to each other,
	
	\item for each $i\in [q], j\in [h]$, the vertex $m_{i,j}$ has a single neighbour $y_{i,j}$ contained in $A_i$, $y_{i,j}$ has degree 2, and all other neighbours of $m_{i,j}$ are contained in $A_1\cup \dots \cup A_{i-1} \cup \{s_j\}$,
	
	\item for each $j\in [h]$, we have $N(s_j)=\{m_{1,j}, \dots , m_{q,j}\}$, and
	
	\item the vertices of $M$ are pairwise at distance at least 6 from each other in $F-S$.
\end{itemize}

Next we will find as vertex-minors large frames within long $q'$-covers with many dangling paths.

\begin{lemma}\label{(q,h)-frame}
	For all pairs of positive integers $q$ and $h$, there exists a positive integer $q'$ with the following property. Let $G$ be a graph containing a long $q'$-cover $({\mathcal{L}}_i : i\in [q'])$ with $h$ paths $P_1,\dots, P_h$ dangling spaciously. Then $G$ contains as a vertex-minor a $(q,h)$-frame $F$.
\end{lemma}

\begin{proof}
	Fix $q$ and $h$. Let $q_0=q$ and for each $j\in [h]$ in order, let $q_j=R(q_{j-1},q_{j-1})$. Let $q'=q_h$.
	
	Firstly by removing vertices we may assume that all vertices of $G$ belong to either the long $q'$-cover $({\mathcal{L}}_i : i\in [q'])$ or one of its paths dangling spaciously.
	Let $M'=\{m_{i,j}' : i\in [q'], j\in[h] \}$ be the vertices of the long $q'$-cover that the paths $P_1,\dots, P_h$ dangle spaciously from. Now for each $m_{i,j}'\in M'$, let $y_{i,j}'$ be a vertex of $L_i^2$ adjacent to $m_{i,j}'$ and let $z_{i,j}'$ be a vertex of $L_i^1$ adjacent to $y_{i,j}'$. Let $Y'=\{y_{i,j}' : i\in [q'], j\in[h] \}$ and $Z' = \{z_{i,j}' : i\in [q'], j\in[h] \}$. Let $G'=G[M'\cup Y' \cup Z' \cup (\bigcup_{i=1}^{q'} L_i^0) \cup (\bigcup_{j=1}^h V(P_j))]$.
	
	Now in $G'$, for each $j\in [h]$, the path $P_j$ dangles spaciously from the set $\{m_{1,j}', \dots , m_{q',j}'\}$.
	Notice that $G'/\bigcup_{i=1}^h E(P_i)$ is a $(q',h)$-frame (where each path $P_j$ is contracted to a single vertex $s_j$, and for each $i\in [q']$ we have $A_i=L_i^0\cup \{z_{i,j}' : j\in [h]\} \cup \{y_{i,j}' : j\in [h]\} )$. However of course $G'/\bigcup_{i=1}^h E(P_i)$ need not be a vertex-minor of $G^*$. But with $h$ applications of Lemma~\ref{contract 1 hanging path}, we see that $G'$ does at least contain a $(q,h)$-frame $F$ as a vertex-minor as we require.
\end{proof}

We say that a $(q,h)$-frame $F$ is \emph{trimmed} if for each $i\in [q]$, the induced subgraph $F[A_i\cup \{m_{i,1},\dots ,m_{i,h}\}]$ is a bloated tree $T_i$ such that

\begin{itemize}
	\item the leaves of $T_i$ are $\{m_{i,1},\dots ,m_{i,h}\}$, and
	
	\item in $T_i$ all big cliques and branching vertices are at distance at least 4 from each other.
\end{itemize}

Given a $(q,h')$-frame with $h'$ sufficiently large we can apply Lemma~\ref{1 L} to each of the connected induced subgraphs $F[A_1],\dots , F[A_q]$ in order (with the distinguished vertices being contained in $\{m_{i,1},\dots ,m_{i,h}\}$ for each $F[A_i]$) to obtain a trimmed $(q,h)$-frame as an induced subgraph.

To consolidate our position, by Lemma~\ref{(q,h)-frame} and repeated application of Lemma~\ref{1 L} as discussed, we obtain the following lemma.

\begin{lemma}\label{trimmed frame}
	For all pairs of positive integers $q$ and $h$, there exists a pair of positive integers $q'$ and $h'$ with the following property. Let $G$ be a graph containing a long $q'$-cover $({\mathcal{L}}_i : i\in [q'])$ with $h'$ paths $P_1,\dots, P_{h'}$ dangling spaciously. Then $G$ contains as a vertex-minor a trimmed $(q,h)$-frame $F$.
\end{lemma}

We call a $(q,h)$-frame \emph{pure} if it is trimmed and in addition, for each $i\in [q]$, the induced bloated tree $T_i=F[A_i\cup \{m_{i,1},\dots ,m_{i,h}\}]$ is shrinkable (as defined in Section~4).

Notice that the bloated trees of a trimmed $(q,h)$-frame are rather close to being shrinkable. For example if $T_i$ is a bloated tree of a trimmed $(q,h)$-frame $F$, then the leaves $\{m_{i,1},\dots ,m_{i,h}\}$ of $T_i$ are at distance at least 2 in $T_i$ from big cliques of $T_i$ since by the definition of a $(q,h)$-frame, the unique neighbour of each vertex of $\{m_{i,1},\dots ,m_{i,h}\}$ in $T_i$ has degree 2.
Additionally if $Z_i$ is the set of non-leaf vertices of $T_i$ that have a neighbour outside the bloated tree $T_i$, then each vertex of $Z_i$ has a neighbour in $M$ and all their neighbours outside the bloated tree are contained in $M$. Since the distance in $F-S$ between vertices of $M$ is at least 6, each vertex in $Z_i$ has just a single neighbour outside $T_i$, and similarly if $z,z'\in Z_i$ have different neighbours outside $T_i$, then the distance between $z$ and $z'$ in $T_i$ is at least 4.
Also in $T_i$, the vertices in $Z_i$ are at distance at least 5 from leaf vertices of $T_i$ since the leaves are contained in $M$.
So a bloated tree $T_i$ of a trimmed $(q,h)$-frame is shrinkable if in $T_i$, the vertices in $Z_i$ are at distance at least 4 from the branching vertices and big cliques of $T_i$.

Therefore an equivalent definition of a pure $(q,h)$-frame is that it is a trimmed $(q,h)$-frame such that for each $i\in [q]$ the vertices in $Z_i$ in the bloated tree $T_i$ are at distance at least 4 in $T_i$ from the branching vertices and big cliques of $T_i$.

The next step is to make trimmed frames pure.

\begin{lemma}\label{pure frame}
	Let $F$ be a trimmed $(q(3h-5),h)$-frame with $h\ge 2$. Then $F$ contains as an induced subgraph a pure $(q,h)$-frame $F'$.
\end{lemma}

\begin{proof}
	First of all we may assume that no bloated tree $T_j$ of $F$ contains a vertex of degree $k-1$ which belongs to a big $k$-clique, as otherwise we may simply delete this vertex and any new leaves this creates as necessary.
	
	We consider a $(q(3h-5))$-vertex auxiliary graph $H$ with vertex set $[q(3h-5)]$ corresponding to the bloated trees $T_1,\dots, T_{q(3h-5)}$. For each $i<j$, $ij$ is an edge of $H$ if and only if in $F$ there is an edge between some leaf of $T_j$ and some vertex $w$ of $T_i$ that is at distance at most 3 in $T_i$ from a big clique or a branching vertex of $T_i$.
	
	Fix $i\in [q(3h-5)]$ and consider a set $C$ which is either a big clique or a single branching vertex of $T_i$. Then take $k=|N(C)\cap V(T_i)|$. Let $J$ be the set of all $j\in [q(3h-5)] - [i]$, such that some leaf of $T_j$ has a neighbour in $T_i$ at distance at most 3 in $T_i$ from $C$.
	Note that the vertices at distance at most 3 from $C$ in $T_i$ form a bloated tree with $C$ being the only big clique or branching vertex.
	So observe that, as the vertices of $M$ are at distance at least 6 from each other in $F-S$, we must have that $|J|\le k$.
	
	In a tree $T$ with $\ell$ leaves, the sum of the degrees of the branching vertices is at most $3(\ell -2)=3\ell -6$.
	Hence in $H$, $i$ is adjacent to at most $3h-6$ vertices $j$ with $j>i$.
	Hence $H$ is $(3h-6)$-degenerate and so $(3h-5)$-colourable. Hence $H$ has a stable set $I$ of size $\frac{q(3h-5)}{3h-5}=q$.
	
	Finally observe that we may obtain a pure $(q,h)$-frame $F$ by deleting the bloated trees $T_j$ such that $j$ is not contained in the stable set $I$.
\end{proof}

Each bloated tree of a pure frame is shrinkable and so Lemma~\ref{contract bloated tree} can be applied. We now prove the last lemma of this section.

\begin{lemma}\label{contracted pure frame}
	Let $H$ be a graph, then there exists positive integers $q'$ and $h'$ such that every graph $G$ containing a long $q'$-cover $({\mathcal{L}}_i : i\in [q'])$ with $h'$ paths $P_1,\dots, P_{h'}$ dangling spaciously contains the graph $H$ as a vertex-minor.
\end{lemma}

\begin{proof}
	Let $q^*,h^*$ be such that every interfered $K_{q^*.h^*}^1$ contains $H$ as a vertex-minor (as in Lemma~\ref{Int K^1_n,n}). By Lemma~\ref{trimmed frame} and Lemma~\ref{pure frame}, there exist positive integers $q'$ and $h'$ such that every graph $G$ containing a long $q'$-cover $({\mathcal{L}}_i : i\in [q'])$ with $h'$ paths $P_1,\dots, P_{h'}$ dangling spaciously contains a pure $(q^*,h^*)$-frame $F$ as a vertex-minor.
	Next we shall show that $F$ contains an interfered $K_{q^*,h^*}^1$ as a vertex-minor.
	
	For each $i,j\in [q^*]$ with $i<j$, let $E_{i,j}'$ be the set of edges between the two sets of vertices $A_i$ and $\{m_{j,1},\dots,m_{j,h^*}\}$ in $F$. Let $E'$ be the union of all such sets $E_{i,j}'$.
	Then by an application of Lemma~\ref{contract bloated tree} to each of the bloated trees $T_1, \dots , T_{q^*}$, there exists a set $E^*$ of edges contained $E'$ such that $F$ contains $(F-E^*)/\bigcup_{i=1}^{q^*} E(A_i) $ as a vertex-minor. This graph is an interfered $K_{q^*,h^*}^1$. So we conclude that $G$ contains $H$ as a vertex-minor as required.
\end{proof}

We may now prove the main result of this section, Theorem~\ref{liner 9-control} that proper vertex-minor-closed classes of graphs are linearly 9-controlled.

\begin{proof}[Proof of Theorem~\ref{liner 9-control}]
	Let $\mathcal{G}$ be a proper vertex-minor-closed class of graphs and let $H$ be a graph not contained in $\mathcal{G}$. Let $q',h'$ be as in Lemma~\ref{contracted pure frame}. We will show that for each $G\in {\mathcal{G}}$, with $\chi^{(9)}(G)\le \kappa$, we have $\chi(G) \le 2^{q'}q'h'\kappa$.
	
	Suppose that $\chi(G) > 2^{q'}q'h'\kappa$. Then by Lemma~\ref{long q-cover h hanging paths}, $G$ contains a long $q'$-cover $({\mathcal{L}}_i : i\in [q'])$ with $h'$ paths $P_1 ,\dots , P_{h'}$ dangling spaciously. But then by Lemma~\ref{contracted pure frame}, the graph $G$ would contain $H$ as a vertex-minor, a contradiction.
	
	Hence $\chi(G) \le 2^{q'}q'h'\kappa$ as required.
\end{proof}

\section{From 9-control to 2-control}

In this section we quickly extend the fact that proper vertex-minor closed classes are 9-controlled to prove that they are in fact 2-controlled.
We make use of a theorem of Chudnovsky, Scott and Seymour~\cite{chudnovsky2016induced}.

\begin{theorem}[Chudnovsky, Scott and Seymour~{\cite[1.10]{chudnovsky2016induced}}]\label{control}
	Let $\mu \ge 0$, and let $\rho \ge 2$. Let $\mathcal{G}$ be a $\rho$-controlled class of graphs that is closed under taking induced subgraphs. Then the class of all graphs in $\mathcal{G}$
	that do not contain any of $K_{\mu,\mu}^1,\dots ,K_{\mu,\mu}^{\rho +2}$ as an induced subgraph is 2-controlled.
\end{theorem}

\begin{theorem}\label{2-control}
	Every proper vertex-minor-closed class of graphs is 2-controlled.
\end{theorem}

\begin{proof}
	Let $\mathcal{G}$ be a proper vertex-minor class of graphs.
	Let $H$ be a graph not contained in $\mathcal{G}$ and let $\mu= {|V(H)| \choose 2}$. Then by smoothing vertices and applying Lemma~\ref{K^1_n,n}, we see that each of the graphs $K_{\mu,\mu}^1,\dots ,K_{\mu,\mu}^{11}$ are not contained in $\mathcal{G}$. Hence by Theorem~\ref{liner 9-control} and Theorem~\ref{control}, it follows that $\mathcal{G}$ is 2-controlled.
\end{proof}

\section{Vertex-minor $\chi$-boundedness}

In this section we prove that if a vertex-minor-closed class of graphs $\mathcal{G}$ is 2-controlled then $\mathcal{G}$ is $\chi$-bounded. As all proper vertex-minor-closed classes of graphs are 2-controlled by Theorem~\ref{2-control}, this implies our main result, Theorem~\ref{main}, that proper vertex-minor-closed classes of graphs are $\chi$-bounded.

Let $G$ be a graph, and $X,C\subseteq V(G)$ such that $X$ has a total ordering $\preceq$. For each $x\in X$, let $N_x\subseteq N(x)$. We say that $(N_x : x\in X)$ is a \emph{multicover}\footnote{Our concept of a multicover is slightly different from that of~\cite{chudnovsky2020induced} and~\cite{scott2019induced}. Their definition of a multicover corresponds to our definition of a pure multicover.} of $C$ if:

\begin{itemize}
	\item the sets $X,C,(N_x : x\in X)$ are disjoint,
	\item the set $X$ is stable,
	\item the set $X$ is anti-complete to $C$,
	\item for each $x\in X$, $N_x$ dominates $C$, and
	\item for distinct $x,y\in X$, with $x\prec y$, the vertex $y$ is anti-complete to $N_x$.
\end{itemize}

If additionally for all distinct $x,y\in X$, the vertex $y$ is anti-complete to $N_x$ then $(N_x : x\in X)$ is a \emph{pure multicover} of $C$.

We say that $(N_x: x\in X)$ is an \emph{impure multicover} if it is a multicover and for each distinct $x,y\in X$, with $x\prec y$, the vertex $x$ is complete to $N_y$.
A multicover is \emph{stable} if for each $x\in X$, the set $N_x$ is stable.
The \emph{length} of a multicover is equal to $|X|$.

We begin now by showing that 
every graph of sufficiently large chromatic number and small clique number in a 2-controlled class of graphs contains a 
large stable  multicover of a set with large chromatic number.

\begin{lemma}\label{pseudo multicover}
	Let $c,\ell,\tau, \omega$ be non-negative integers and let $\mathcal{G}$ be a 2-controlled class of graphs that is closed under taking induced subgraphs such that $\chi(G)\le \tau$ for all $G\in {\mathcal{G}}$ with $\omega(G)<\omega$.
	Then there exists a positive integer $c'$ such that every graph $G\in {\mathcal{G}}$ with $\chi(G)\ge c'$ and $\omega(G)\le \omega$ contains a length-$\ell$ stable multicover of a set $C\subseteq V(G)$ with $\chi(C)\ge c$.
\end{lemma}

\begin{proof}
	We fix $c,\tau, \omega$.
	The result is trivial for $\ell=0$, so we proceed inductively assuming the result holds for $\ell-1$.
	Let $c_0'$ be an integer such that every graph $G\in {\mathcal{G}}$ with $\chi(G)\ge c_0'$ and $\omega(G)\le \omega$ contains a length-$(\ell -1)$ stable multicover of a set $C'$ with chromatic number at least $c$.
	Now let $c'$ be such that every graph $G\in {\mathcal{G}}$ with $\chi(G)\ge c'$ contains a 2-ball with chromatic number at least $\tau + \tau c'_0$. It remains to show that $c'$ satisfies the conclusion of the lemma.
	
	Let $G$ be a graph in $\mathcal{G}$ with $\chi(G) \ge c'$ and $\omega(G) \le \omega$.
	Let $y$ be a vertex of $G$ such that $N_{2}[y]$ has chromatic number at least $\tau + \tau c'_0$. Then $\chi(N(y))\le \tau$, so $\chi(N_2(y))\ge \tau c'_0$. Let $N_{y}$ be a stable subset of $N(y)$ such that ${\chi(N(N_{y}) \cap N_2(y))\ge c'_0}$.
	So by the induction hypothesis, there exists a set $C$ contained in $N(N_{y}) \cap N_2(y)$ with $\chi(C)\ge c$ and a stable multicover $(N_{x'}:x'\in X')$ of length $\ell-1$ in $G[N(N_{y}) \cap N_2(y)]$ of the set $C$.
	
	Let $X=X'\cup \{y\}$ and let $\preceq$ be the total ordering on $X$ where $y$ is the largest vertex and the restriction to $X'$ is $\preceq'$.
	Then $(N_x:x\in X)$ provides the desired stable multicover of $C$.
\end{proof}

Next we wish to obtain either a pure or an impure multicover.

\begin{lemma}\label{impure multicover}
	Let $c$ and $m$ be positive integers, and let $m'=R(m,m)$.
	If a graph $G$ has a length-$m'$ stable multicover $(N_x': x\in X')$ of a set $C$ with $\chi(C)\ge c2^{m'\choose 2}$, then it has a length-$m$ stable multicover $(N_x: x\in X)$ of a set $C'\subseteq C$ with $\chi(C')\ge c$ that is either pure or impure.
\end{lemma}

\begin{proof}
	For each $v\in C$ and $x\in X'$, let $p_v^x \in N_x'$ be some vertex adjacent to $v$.
	For each $v\in C$, let $f(v)$ be the auxiliary graph on vertex set $X'$ such that for each pair $x,y\in X'$ with $x \prec y$, $x$ is adjacent to $y$ in $f(v)$ if $p_v^y$ is adjacent to $x$ in $G$.
	By the pigeonhole principle there exists some graph $H$ on vertex set $X'$ such that $f^{-1}(H)\subseteq C$ has chromatic number at least $c$.
	Let $C' = f^{-1}(H)$.
	
	Now for each  pair $x,y\in X$ with $x \prec y$, the set $P_{y}=\{p_v^y : v\in V(C')\}$ is either complete or anti-complete to $x$ depending on whether or not $x$ is adjacent to $y$ in $H$. Since $|X'|=m'=R(m,m)$, there exists a set $X\subseteq X'$ with $|X|=m$ that is either stable in $H$ or a clique in $H$. For each $x\in X$, let $N_x=P_x$, then $(N_x: x\in X)$ is a stable multicover of $C'$. Furthermore, if $X'$ is a stable set in $H$, then $(N_x: x\in X)$ is a pure multicover of $C'$ and if $X'$ is a clique in $H$, then $(N_x: x\in X)$ is an impure multicover of $C'$ as required.
\end{proof}

We wish to find a stable pure multicover rather than a stable impure multicover. To this end we next show that large stable impure multicovers contain large stable pure multicovers as pivot-minors (and therefore as vertex-minors).

\begin{lemma}\label{multicover}
	For all positive integers $c,\ell,\tau, \omega, \omega^*$ such that $\omega^*\le \omega$, there exists a pair of positive integers $c',\ell'$ with the following property.
	Let $\mathcal{G}$ be a pivot-minor-closed class of graphs such that all $H\in {\mathcal{G}}$ with $\omega(H) < \omega$ have $\chi(H) \le \tau$.
	Let $G\in {\mathcal{G}}$ be a graph with $\omega(G)\le \omega$ that contains a length-$\ell'$ stable impure multicover $(N_x : x\in X)$ of a set $C$ such that $\chi(C)\ge c'$ and $\omega\left(\bigcup_{x\in X}N_x\right)\le \omega^*$. Then $G$ contains as a pivot-minor a graph $G^*$ with $\omega(G^*)\le \omega$ that contains a length-$\ell$ stable pure multicover of a set $C'$ with $\chi(C')\ge c$.
\end{lemma}

\begin{proof}
	The case $\ell=1$ is trivial, so we may assume that $\ell \ge 2$. We fix $c,\ell,\tau, \omega$ and will now proceed by induction on $\omega^*$.
	
	First we handle the case that $\omega^*=1$, or in other words the case that $\bigcup_{x\in X} N_x$ is a stable set.
	In this case let $c'=c+\ell \tau$ and $\ell '= \ell$. Let $X=\{x_1,\dots ,x_{\ell'}\}$ where $x_1 \prec \dots \prec x_{\ell'}$. For each $x\in X$, let $n_x$ be a vertex of $N_x$ and let $C_x=N(n_x)\cap C$.
	By deleting vertices we may assume that $V(G)= C \cup X \cup \bigcup_{x\in X} N_x$.
	Then let $G^*=(G\wedge x_{\ell'}n_{x_{\ell'}} \wedge \dots \wedge x_1n_{x_1}) - \left( X \cup \bigcup_{x\in X} C_x \right)$.
	Equivalently, $G^*$ is the graph obtained from $G$ by
	removing the edges between $x$ and $N_y-\{n_y\}$ for each pair of distinct vertices $x,y\in X$, then
	deleting the vertices $\bigcup_{x\in X} \left( \{n_x\} \cup C_x \right)$, and then
	for each $x\in X$, relabelling the vertex $x$ as $n_x$.
	So $\omega(G^*)\le \omega(G) \le \omega$. Let $C'= C - \bigcup_{x\in X} C_x$, then $\chi(C') \ge \chi(C) - \ell\tau \ge c$. Let $X'=\{n_x : x\in X\}$ and for each $x\in X$, let $N_{n_x}'=N_x -\{n_x\}$.
	Then ${(N_{x'}' : x'\in X')}$ provides the desired stable pure multicover of $C'$.
	
	Now let us assume that $\omega^* > 1$ and the inductive hypothesis holds for $\omega^*-1$.
	Let $c'_0,\ell'_0$ be as obtained by the inductive hypothesis for $c,\ell,\tau,\omega,\omega^*-1$.
	Let $\ell'= \ell_0'^2 R(\omega^*+1, \ell)$ and $c' = 4^{\ell' \choose 2}c'_0 + \ell'\tau$.
	
	As before, for each $x\in X$, let $n_x$ be a vertex of $N_x$. Then for each $x\in X$, let $C_x=N(n_x)\cap C$.
	Let $C_0' = C - \bigcup_{x\in X} C_x$, then $\chi(C'_0) \ge c' - \ell'\tau=4^{\ell' \choose 2}c'_0$.
	
	Now for each $v\in C'_0$ and each $x\in X$, let $p_v^x\in N_x$ be a vertex adjacent to $v$. For each $v\in C'_0$, let $f_1(v)$ be the auxiliary graph on vertex set $X$ such that for each pair of distinct vertices  $x,y\in X$ with $x\prec y$, $y$ is adjacent to $x$ in $f_1(v)$ if $p_v^y$ is adjacent to $n_x$.
	Similarly for each $v\in C'_0$, let $f_2(v)$ be the auxiliary graph on vertex set $X$ such that for each pair of distinct vertices  $x,y\in X$ with $x\succ y$, $y$ is adjacent to $x$ in $f_2(v)$ if $p_v^y$ is adjacent to $n_x$.
	By pigeonhole principle there exists some pair of graphs $H_1$ and $H_2$ on the vertex set $X$ such that $f_1^{-1}(H_1) \cap f_2^{-1}(H_2) \subseteq C_0'$ has chromatic number at least $\chi(C_0')/ 4^{|X| \choose 2} =\chi(C_0') / 4^{\ell' \choose 2} \ge c_0'$.
	Let $C'_1=f_1^{-1}(H_1) \cap f_2^{-1}(H_2)$.
	Now for each distinct pair $x,y\in X$, the set $P_{y}=\{p_v^y : v\in C'_1\}$ is either complete or anti-complete to $n_x$.
	
	Suppose that $H_1$ contains a vertex $x\in X$ such that $x$ has at least $\ell_0'$ neighbours $y$ in $H_1$ with $x \prec y$.
	Let $Y$ be the set of neighbours $y$ of $x$ in $H_1$ with $x \prec y$.
	Then $|Y|\ge \ell_0'$ and $x$ is complete to $\bigcup_{y\in Y} P_y$ in $G$.
	Now $\omega(\bigcup_{y\in Y} P_y)<\omega^*$, so by the inductive hypothesis, there exists a pivot-minor $G^*$ of $G$ containing a length-$\ell$ stable pure multicover of a set $C'$ with $\chi(C')\ge c$ as required. Hence we may assume that $H_1$ contains no vertex $x\in X$ such that $x$ has at least $\ell_0'$ neighbours $y$ in $H_1$ with $x \prec y$.
	So $H_1$ is $(\ell_0'-1)$-degenerate, and so $\ell_0'$-colourable.
	Similarly $H_2$ is also $\ell_0'$-colourable.
	
	Let $H$ be the graph on vertex set $X$ with edge set $E(H_1)\cup E(H_2)$.
	Then $H$ is $\ell_0'^2$-colourable.
	So $H$ contain a stable set $Y_0$ of size at least $\ell'/\ell_0'^2 = R(\omega^*+1, \ell)$.
	Since $\omega(\bigcup_{y\in Y_0} P_y)\le \omega^*$, there exists a subset $Y$ of $Y_0$ of size $\ell$ such that the set $\{n_y : y\in Y\}$ is stable in $G$. Let $Y=\{y_1,\dots , y_{\ell}\}$ with $y_1 \prec \dots \prec y_{\ell}$.
	Then let $G_0^*$ be the subgraph of $G$ induced on $C'\cup Y \cup \bigcup_{y\in Y} \left(  P_y \cup \{n_y\}  \right)$.
	Let $G^*=G_0^*\wedge y_{\ell}n_{y_{\ell}} \wedge \dots \wedge y_1n_{y_1}$.
	Similarly to before, $G^*$ is equivalently the graph obtained from $G_0^*$ by
	removing the edges between $x$ and $N_y-\{n_y\}$ for each pair of distinct vertices $x,y\in Y$, and then
	for each $y\in Y$, swapping the labels of $y$ and $n_y$.
	So $\omega(G^*)\le \omega(G^*_0) \le \omega(G) \le \omega$.
	Let $Y'=\{n_y : y\in Y\}$ and for each $y\in Y$, let $P_{n_y}'=P_y -\{n_y\}$.
	Finally $(P_{y'}' : y'\in Y')$ provides our desired stable pure multicover of $C'$.
\end{proof}

We consolidate our position, and by Lemmas~\ref{pseudo multicover},~\ref{impure multicover} and~\ref{multicover}, we obtain the following:

\begin{lemma}\label{pivot-minor multicover}
	Let $c,\ell,\tau, \omega$ be positive integers and let $\mathcal{G}$ be a 2-controlled class of graphs closed under pivot-minors such that $\chi(H)\le \tau$ for all $H\in {\mathcal{G}}$ with $\omega(H)<\omega$.
	Then there exists a positive integer $c'$ such that
	every graph $G\in {\mathcal{G}}$ with $\omega(G)\le \omega$ and $\chi(G)\ge c'$ contains as a pivot-minor a graph $G^*$ with $\omega(G^*)\le \omega$ such that $G^*$ contains a length-$\ell$ stable pure multicover of a set $C$ with $\chi(C)\ge c$.
\end{lemma}

Let $G$ be a graph and let $(N_x :x\in X)$ be a pure multicover of a set $C\subseteq V(G)$. We say that the pure multicover is \emph{stably $k$-crested} if $V(G)-X-C- \left(\bigcup_{x\in X} N_x \right)$ has distinct elements $a_1,\dots, a_k$ and $a_{i,x}$ for all $i\in [k]$ and $x\in X$ such that:

\begin{itemize}
	\item for each $i\in [k]$, $N(a_i)\cap \{a_{j,x} : j\in [k], x\in X\} = \{a_{i,x}: x\in X\}$,
	
	\item for each $i\in [k]$ and $x\in X$, $N(a_{i,x})\cap (C\cup X \cup \bigcup_{y\in X} N_y) = \{x\}$,
	
	\item $\{a_1,\dots , a_k\}$ is anti-complete to $C\cup X \cup \bigcup_{y\in X} N_y$, and
	
	\item both $\{a_1,\dots , a_k\}$ and $\{a_{i,x} : i\in [k], x\in X\}$ are stable.
\end{itemize}

We call the vertices of $\{a_1,\dots, a_k\}$, \emph{centres} of a stably $k$-crested pure multicover.
Chudnovsky, Scott, Seymour and Spirkl~\cite{chudnovsky2020induced} proved that:

\begin{theorem}[Chudnovsky, Scott, Seymour and Spirkl~{\cite[4.4]{chudnovsky2020induced}}]\label{k-crested}
	For all positive integers $c,\ell, k, \tau, \omega$, there exists a paur of positive integers $\ell',c' $ with the following property.
	Let $G$ be a graph with $\omega(G) \le \omega$, such that $\chi(H) \le \tau$ for every induced subgraph $H$ of $G$ with $\omega(H) < \omega$. Let $(N_x': x\in X')$ be a pure multicover of some set $C'$ such that $|X'|\ge \ell'$ and $\chi(C)\ge c'$. Then there exists a stably $k$-crested length-$\ell$ stable pure multicover of a set $C$ with $\chi(C)\ge c$.
\end{theorem}

By combining Lemma~\ref{pivot-minor multicover} and Theorem~\ref{k-crested}, we obtain the following (which is also where we will pickup from in Section~10 to prove Theorem~\ref{pivot step}):

\begin{lemma}\label{pivot-minor k-crested}
	Let $c,\ell,k,\tau, \omega$ be positive integers and let $\mathcal{G}$ be a 2-controlled class of graphs closed under pivot-minors such that $\chi(H)\le \tau$ for all $H\in {\mathcal{G}}$ with $\omega(H)<\omega$.
	Then there exists a positive integer $c'$ such that
	every graph $G\in {\mathcal{G}}$ with $\omega(G)\le \omega$ and $\chi(G)\ge c'$ contains as a pivot-minor a graph $G^*$ with $\omega(G^*)\le \omega$ such that $G^*$ contains a stably $k$-crested length-$\ell$ stable pure multicover of a set $C$ with $\chi(C)\ge c$.
\end{lemma}

Notice in particular that if a graph $G$ contains a stably $k$-crested stable pure multicover $(N_x : x\in X)$, with $|X|=\ell$, then $G$ contains an induced $K_{\ell,k}^1$.

\begin{theorem}\label{vertex-minor 2-control to bounded}
	Every vertex-minor-closed class of graphs that is 2-controlled is also $\chi$-bounded.
\end{theorem}

\begin{proof}
	Let $\mathcal{G}$ be a 2-controlled vertex-minor closed class of graphs and suppose for sake of contradiction that $\mathcal{G}$ is not $\chi$-bounded.
	Then there exists a minimum integer $\omega \ge 2$ such that graphs $G\in {\mathcal{G}}$ with $\omega(G) \le \omega$ have unbounded chromatic number.
	Let $\tau \ge 0$ be such that $\chi(G)\le \tau$ for all $G\in {\mathcal{G}}$ with $\omega(G)<\omega$. Let $H$ be some graph not contained in $\mathcal{G}$ ($H$ exists as the class of all graphs is not 2-controlled). Let $\ell = |V(H)|$ and $k={|V(H)| \choose 2}$. Let $c=0$. Then let $c'$ be as in the conclusion of Lemma~\ref{pivot-minor k-crested} for $c,\ell,k,\tau, \omega$. Then there must exist a graph $G\in {\mathcal{G}}$ with $\omega(G) \le \omega$ and $\chi(G)\ge c'$. But by Lemma~\ref{pivot-minor k-crested}, $G$ must contain $K_{\ell,k}^1$ as a pivot-minor, and so by Lemma~\ref{K^1_n,n}, $G$ must contain $H$ as a vertex-minor, a contradiction.
\end{proof}

Now proving the main result, Theorem~\ref{main}, (which states that proper vertex-minor-closed classes of graphs are $\chi$-bounded) is straightforward.

\begin{proof}[Proof of Theorem~\ref{main}]
	Let $\mathcal{G}$ be a proper vertex-minor-closed class of graphs.
	By Theorem~\ref{2-control}, $\mathcal{G}$ is 2-controlled. So then by Theorem~\ref{vertex-minor 2-control to bounded}, $\mathcal{G}$ is $\chi$-bounded.
\end{proof}

\section{Pivot-minors and a step towards $\chi$-boundedness}

In this section we make a first step towards proving that proper pivot-minor-closed classes of graphs are $\chi$-bounded. In particular we will prove Theorem~\ref{pivot step}, the pivot-minor analogue of Theorem~\ref{vertex-minor 2-control to bounded}. We may continue from where Lemma~\ref{pivot-minor k-crested} left off in Section~9.

Let $G$ be a graph and let $(N_x :x\in X)$ be a pure multicover of a set $C\subseteq V(G)$. We say that an induced path $P$ is an \emph{oddity} for the pure multicover if:

\begin{itemize}
	\item $P$ has length $3$ or $5$,
	
	\item the ends of $P$ are in $X$,
	
	\item no vertex of $X$ that is not an end of $P$ has a neighbour or is contained in $V(P)$, and
	
	\item $V(P)\subseteq C \cup X \cup  \bigcup_{x\in X} N_x $.
\end{itemize}

Scott and Seymour~\cite{scott2019induced} proved the following:

\begin{theorem}[Scott and Seymour~{\cite[2.2]{scott2019induced}}]\label{oddities}
	For all positive integers $n, \tau, \omega$, there exist positive integers $\ell, c$ with the following property. Let $G$ be a graph such that $\omega(G)\le \omega$, for every induced subgraph $H$ of $G$ with $\omega(H)<\omega$, we have $\chi(H)\le \tau$, and $G$ contains a stable pure multicover $(N_x :x\in X)$ of length $\ell$ of a set $C$, where $\chi(C)\ge c$.
	Then the multicover contains $n$ vertex-disjoint oddities $P_1,\dots , P_n$, where $V(P_1),\dots , V(P_n)$ are pairwise anti-complete.
\end{theorem}

A graph $H$ is a \emph{proper odd subdivision} of a graph $H'$ if $H$ can be obtained from $H'$ by replacing each edge with a path of odd length at least 3.
Notice that if a graph $G$ contains a stably $n$-crested stable pure multicover $(N_x : x\in X)$ of a set $C$, such that the multicover contains $n \choose 2$ vertex-disjoint oddities $P_1,\dots , P_{n\choose 2}$, where $V(P_1), \dots V(P_{n\choose 2})$ are pairwise anti-complete, then $G$ contains a proper odd subdivision of $K_n$ (where the vertices of $K_n$ are the centre vertices of the stably $n$-crested stable pure multicover $(N_x : x\in X)$). If $u,v,w,x$ are vertices of a graph $G$ such that $N(v)=\{u,w\}$ and $N(w)=\{v,x\}$, then $(G\wedge vw)-v-w$ is isomorphic to the graph $G/\{vw,wx\}$.
Hence a proper odd subdivision of $K_n$ contains any $n$-vertex graph as a pivot-minor. So by Lemma~\ref{pivot-minor k-crested} and Theorem~\ref{oddities}, we may obtain the following:

\begin{lemma}\label{n-vertex pivot-minor}
	Let $\tau, \omega$ be positive integers, let $J$ be a graph, and let $\mathcal{G}$ be a 2-controlled class of graphs closed under pivot-minors such that $\chi(H)\le \tau$ for all $H\in {\mathcal{G}}$ with $\omega(H)<\omega$.
	Then there exists a positive integer $c$ such that
	every graph $G\in {\mathcal{G}}$ with $\omega(G)\le \omega$ and $\chi(G)\ge c$ contains $J$ as a pivot-minor.
\end{lemma}

Then a simple induction on $\omega$, making use of Lemma~\ref{n-vertex pivot-minor}, proves Theorem~\ref{pivot step}, that 2-controlled pivot-minor-closed classes are $\chi$-bounded:

\begin{proof}[Proof of Theorem~\ref{pivot step}]
	Let $\mathcal{G}$ be 2-controlled class of graphs that is closed under pivot-minors.
	Suppose for the sake of contradiction that $\mathcal{G}$ is not $\chi$-bounded. Then there exists a minimum integer $\omega \ge 2$ such that graphs $G\in {\mathcal{G}}$ with $\omega(G) \le \omega$ have unbounded chromatic number.
	Let $\tau \ge 0$ be such that $\chi(G)\le \tau$ for all $G\in {\mathcal{G}}$ with $\omega(G)<\omega$. Let $J$ be some graph not contained in $\mathcal{G}$ ($J$ exists because the class of all graphs is not 2-controlled). Then let $c$ be as in the conclusion of Lemma~\ref{n-vertex pivot-minor} for $\tau, \omega$ and $J$. Then there must exist a graph $G\in {\mathcal{G}}$ with $\omega(G) \le \omega$ and $\chi(G)\ge c$. But then $G$ must contain $J$ as a pivot-minor, contradicting the fact that $J\not\in {\mathcal{G}}$. 
\end{proof}

\section*{Acknowledgements}

The topic of the authors master's thesis~\cite{thesis} was also on colouring graphs with a forbidden vertex-minor and some of the ideas presented in this paper originate from there. The author would like to thank Dan Kr\'{a}l' for his supervision during the master's and Louis Esperet for sharing the problem that the thesis was based on. The author would additionally like to thank Louis Esperet and R\'{e}mi de Joannis de Verclos for sharing the satisfying proof of Lemma~\ref{branching} yielding an improved bound. The author also thanks Rose McCarty for suggesting a number of improvements on an early draft of this manuscript. Lastly, the author thanks the anonymous referees for thorough and very helpful reports that provided a number of much needed corrections.

\bibliographystyle{abbrv}

\begin{thebibliography}{10}
	
	\bibitem{bonamy2019graphs}
	M.~Bonamy and M.~Pilipczuk.
	\newblock Graphs of bounded cliquewidth are polynomially $\chi$-bounded.
	\newblock {\em Advances in Combinatorics}, 2020:8, 21pp.
	
	\bibitem{bouchet1994circle}
	A.~Bouchet.
	\newblock Circle graph obstructions.
	\newblock {\em Journal of Combinatorial Theory, Series B}, 60(1):107--144, 1994.
	
	\bibitem{choi2019chi}
	H.~Choi, O.~Kwon, S.~Oum, and P.~Wollan.
	\newblock Chi-boundedness of graph classes excluding wheel vertex-minors.
	\newblock {\em Journal of Combinatorial Theory, Series B}, 135:319--348, 2019.
	
	\bibitem{choi2017coloring}
	I.~Choi, O.~Kwon, and S.~Oum.
	\newblock Coloring graphs without fan vertex-minors and graphs without cycle
	pivot-minors.
	\newblock {\em Journal of Combinatorial Theory, Series B}, 123:126--147, 2017.
	
	\bibitem{chudnovsky2016induced}
	M.~Chudnovsky, A.~Scott, and P.~Seymour.
	\newblock Induced subgraphs of graphs with large chromatic number. V.
	Chandeliers and strings.
	\newblock {\em Journal of
		Combinatorial Theory, Series B}, 150:195--243, 2021.
	
	\bibitem{chudnovsky2020induced}
	M.~Chudnovsky, A.~Scott, P.~Seymour, and S.~Spirkl.
	\newblock Induced subgraphs of graphs with large chromatic number. VIII. Long
	odd holes.
	\newblock {\em Journal of Combinatorial Theory, Series B}, 140:84--97, 2020.
	
	\bibitem{thesis}
	J.~Davies.
	\newblock {Coloring vertex-minor-free graphs}.
	\newblock Master's thesis, University of Warwick, 2019.
	
	\bibitem{davies2019circle}
	J.~Davies and R.~McCarty.
	\newblock Circle graphs are quadratically $\chi$-bounded.
	\newblock {\em Bulletin of the London Mathematical Society}, 53(3):673--679, 2021.
	
	\bibitem{dvovrak2012classes}
	Z.~Dvo{\v{r}}{\'a}k and D.~Kr{\'a}l'.
	\newblock Classes of graphs with small rank decompositions are $\chi$-bounded.
	\newblock {\em European Journal of Combinatorics}, 33(4):679--683, 2012.
	
	\bibitem{erdos1986maximum}
	P.~Erd\H{o}s, M.~Saks, and V.~S{\'o}s.
	\newblock Maximum induced trees in graphs.
	\newblock {\em Journal of Combinatorial Theory, Series B}, 41(1):61--79, 1986.
	
	\bibitem{branchingEsperet}
	L.~Esperet and R.~de~Joannis~de Verclos.
	\newblock {Personal Communication}, 2018.
	
	\bibitem{fox2009large}
	J.~Fox, P.~S.~Loh, and B.~Sudakov.
	\newblock Large induced trees in $K_r$-free graphs.
	\newblock {\em Journal of Combinatorial Theory, Series B}, 99(2):494--501,
	2009.
	
	\bibitem{Jim}
	J.~Geelen.
	\newblock {Personal Communication}, 2020.
	
	\bibitem{geelen2019grid}
	J.~Geelen, O.~Kwon, R.~McCarty, P.~Wollan.
	\newblock The grid theorem for vertex-minors.
	\newblock {\em Journal of Combinatorial Theory, Series B}, doi:10.1016/j.jctb.2020.08.004, 2020.
	
	\bibitem{golumbic2004algorithmic}
	M.~Golumbic.
	\newblock {\em Algorithmic graph theory and perfect graphs}.
	\newblock second ed., vol. 57, Elsevier Science B.V., Amsterdam, 2004.
	
	\bibitem{gyarfas1985chromatic}
	A.~Gy{\'a}rf{\'a}s.
	\newblock On the chromatic number of multiple interval graphs and overlap
	graphs.
	\newblock {\em Discrete mathematics}, 55(2):161--166, 1985.
	
	\bibitem{gyarfas1985problems}
	A.~Gy{\'a}rf{\'a}s.
	\newblock {\em Problems from the world surrounding perfect graphs}.
	\newblock Number 177. MTA Sz{\'a}m{\'\i}t{\'a}stechnikai {\'e}s
	Automatiz{\'a}l{\'a}si Kutat{\'o} Int{\'e}zet, 1985.
	
	\bibitem{kim2020classes}
	R.~Kim, O.~Kwon, S.~Oum, and V.~Sivaraman.
	\newblock Classes of graphs with no long cycle as a vertex-minor are
	polynomially $\chi$-bounded.
	\newblock {\em Journal of Combinatorial Theory, Series B}, 140:372--386, 2020.
	
	\bibitem{kostochka2004coloring}
	A.~Kostochka.
	\newblock Coloring intersection graphs of geometric figures with a given clique
	number.
	\newblock {\em Contemporary mathematics}, 342:127--138, 2004.
	
	\bibitem{kostochka1988upper}
	A.~Kostochka.
	\newblock Upper bounds on the chromatic number of graphs.
	\newblock {\em Transactions of the Institute of Mathematics (Siberian Branch of the Academy of Sciences of USSR)}, 10 (1988), 204-226 (in Russian).
	
	\bibitem{oum2005rank}
	S.~Oum.
	\newblock Rank-width and vertex-minors.
	\newblock {\em Journal of Combinatorial Theory, Series B}, 95(1):79--100, 2005.
	
	\bibitem{pawlik2014triangle}
	A.~Pawlik, J.~Kozik, T.~Krawczyk, M.~Laso{\'n}, P.~Micek, W.~T. Trotter, and
	B.~Walczak.
	\newblock Triangle-free intersection graphs of line segments with large
	chromatic number.
	\newblock {\em Journal of Combinatorial Theory, Series B}, 105:6--10, 2014.
	
	\bibitem{scott1997induced}
	A.~Scott.
	\newblock Induced trees in graphs of large chromatic number.
	\newblock {\em Journal of Graph Theory}, 24(4):297--311, 1997.
	
	\bibitem{scott2019induced}
	A.~Scott and P.~Seymour.
	\newblock Induced subgraphs of graphs with large chromatic number. X. Holes of
	specific residue.
	\newblock {\em Combinatorica}, 39(5):1105--1132, 2019.
	
	\bibitem{scott2018survey}
	A.~Scott and P.~Seymour.
	\newblock A survey of $\chi$-boundedness.
	\newblock {\em Journal of Graph Theory}, 95(3):473--504, 2020.
	
\end{thebibliography}

\end{document}